\documentclass{amsart}
\usepackage{amsmath, amssymb, amsthm, color, comment}
\usepackage{tikz, here}
\usetikzlibrary{cd}
\usepackage{hyperref}
\usepackage{cleveref}

\newtheorem{dfn}{Definition}[section]
\newtheorem{thm}[dfn]{Theorem}

\newtheorem{lem}[dfn]{Lemma}
\newtheorem{cor}[dfn]{Corollary}

\newtheorem{rem}[dfn]{Remark}

\newtheorem{conj}[dfn]{Conjecture}
\newtheorem{ques}[dfn]{Question}
\newtheorem{claim}{Claim}[dfn]
\newtheorem{subclaim}{Subclaim}[claim]

\crefname{dfn}{Definition}{Definitions}
\crefname{thm}{Theorem}{Theorems}
\crefname{fact}{Fact}{Facts}
\crefname{lem}{Lemma}{Lemmas}
\crefname{cor}{Corollary}{Corollaries}
\crefname{prop}{Proposition}{Propositions}
\crefname{rem}{Remark}{Remarks}
\crefname{notation}{Notation}{Notations}
\crefname{claim}{Claim}{Claims}
\crefname{subclaim}{Subclaim}{Subclaims}
\crefname{conj}{Conjecture}{Conjectures}
\crefname{ques}{Question}{Questions}

\renewcommand{\subset}{\subseteq}
\newcommand{\power}{\wp}
\newcommand{\uphar}{\mathbin{\upharpoonright}}

\DeclareMathOperator{\dom}{dom}
\DeclareMathOperator{\cf}{cf}
\DeclareMathOperator{\crit}{crit}
\DeclareMathOperator{\lh}{lh}
\DeclareMathOperator{\ord}{Ord}

\title{Chang models over derived models with supercompact measures}

\author{Takehiko Gappo}
\address{Takehiko Gappo, Institut f\"{u}r Diskrete Mathematik und Geometrie, TU Wien, Wiedner Hauptstra{\ss}e 8-10/104, 1040 Wien, Austria.}
\email{takehiko.gappo@tuwien.ac.at}

\author{Sandra M\"uller}
\address{Sandra M\"uller, Institut f\"ur Diskrete Mathematik und Geometrie, TU Wien, Wiedner Hauptstra{\ss}e 8-10/104, 1040 Wien, Austria.}
\email{sandra.mueller@tuwien.ac.at}

\author{Grigor Sargsyan}
\address{Grigor Sargsyan, IMPAN, Antoniego Abrahama 18, 81-825 Sopot, Poland.}
\email{gsargsyan@impan.pl}

\subjclass[2020]{03E60, 03E45, 03E55, 03E35}
\keywords{Determinacy, inner model theory, Chang model, derived models, supercompact measures.}
\date{\today}

\begin{document}

\maketitle

\begin{abstract}
Based on earlier work of the third author, we construct a Chang-type model with supercompact measures extending a derived model of a given hod mouse with a regular cardinal $\delta$ that is both a limit of Woodin cardinals and a limit of ${<}\delta$-strong cardinals.
The existence of such a hod mouse is consistent relative to a Woodin cardinal that is a limit of Woodin cardinals.
We argue that our Chang-type model satisfies $\mathsf{AD}_{\mathbb{R}} + \Theta$ is regular + $\omega_1$ is ${<}\delta_{\infty}$-supercompact for some regular cardinal $\delta_{\infty}>\Theta$.
This complements Woodin's generalized Chang model, which satisfies $\mathsf{AD}_{\mathbb{R}}+\omega_1$ is supercompact, assuming a proper class of Woodin cardinals that are limits of Woodin cardinals.
\end{abstract}

\section{Introduction}

The significance of the Axiom of Determinacy ($\mathsf{AD}$) has been amplified through its interactions with descriptive set theory, forcing theory, and inner model theory.
As $\mathsf{AD}$ is an axiom about sets of reals, typical models of $\mathsf{AD}$ are of the form $V=L(\power(\mathbb{R}))$.
In such models, there is no interesting structure above $\Theta$, which is the least ordinal that is not a surjective image of $\mathbb{R}$.
This paper, however, focuses on determinacy models with rich structure above $\Theta$.
We provide a new canonical construction of determinacy models with supercompact measures witnessing that $\omega_1$ is supercompact up to some cardinal above $\Theta$.

\subsection{Motivation behind higher models of determinacy}

Recent groundbreaking results obtained by forcing over determinacy models motivate the study of determinacy models that are not of the form $V=L(\power(\mathbb{R}))$.
Let $\Theta\mathsf{reg}$ denote the theory $\mathsf{ZF}+\mathsf{AD}_{\mathbb{R}}+$ ``$\Theta$ is regular.''\footnote{Note that $\Theta\mathsf{reg}$ implies $\mathsf{AD}^+$, technical strengthening of $\mathsf{AD}$ introduced by Woodin. $\mathsf{AD}^+$ is defined as the conjunction of $\mathsf{DC}_{\mathbb{R}}$, ordinal determinacy, and $\infty$-Borelness of all sets of reals. See \cite{ADplusbook} for the basic theory of $\mathsf{AD}^+$.}
This theory deserves special attention among numerous determinacy theories in the context of Woodin's $\mathbb{P}_{\mathrm{max}}$ forcing.
Our starting point is the following result.

\begin{thm}[Woodin, \cite{Wo10}]\label{Forcing_MM^++(c)}
    Assume that $V=L(\power(\mathbb{R}))$ and $\Theta\mathsf{reg}$ holds.
    If $G\subset\mathbb{P}_{\mathrm{max}}*\mathrm{Add}(\omega_3, 1)$ is $V$-generic, then $V[G]\models\mathsf{ZFC}+\mathsf{MM}^{++}(\mathfrak{c})$.
\end{thm}

Here, $\mathsf{MM}^{++}(\mathfrak{c})$ denotes $\text{Martin's Maximum}^{++}$ for posets of size at most continuum.
We note that \cref{Forcing_MM^++(c)} drastically reduces an upper bound of the consistency strength of $\mathsf{MM}^{++}(\mathfrak{c})$.
Any known way to force $\mathsf{MM}^{++}(\mathfrak{c})$ over $\mathsf{ZFC}$ models requires a supercompact cardinal, while the consistency strength of $\Theta\mathsf{reg}$ is below a Woodin limit of Woodin cardinals (\cite{HMMSC}).

To force more fragments of $\mathsf{MM}^{++}$ via $\mathbb{P}_{\mathrm{max}}$ forcing, one needs to find more complicated determinacy models as potential ground models that may not satisfy $V=L(\power(\mathbb{R}))$.
For example, Blue, Larson, and the third author extended the result of \cite{PmaxSquare} to obtain the following.

\begin{thm}[Blue--Larson--Sargsyan, \cite{Nairian}]\label{BlueLarsonSargsyan}
    Let $3\leq n<\omega$.
    Then it is consistent relative to a Woodin limit of Woodin cardinals that there is a transitive model $M$ of $\Theta\mathsf{reg}$ such that if $G\subset(\mathbb{P}_{\mathrm{max}}*\mathrm{Add}(\omega_3, 1)* \cdot *\mathrm{Add}(\omega_n, 1))^M$ is $M$-generic, then
    \[
    M[G]\models\mathsf{MM}^{++}(\mathfrak{c})+\forall i\in [2, n]\,\neg\square(\omega_i).
    \]
\end{thm}

This result has a striking inner model theoretic corollary: by \cite{JSSS09}, the iterability conjecture for $K^c$ is false in $M[G]$ in the setting of \cref{BlueLarsonSargsyan}.
The conjecture was expected to be a consequence of $\mathsf{ZFC}$ because if so, the construction of a canonical inner model with large cardinals would have been accomplished at least up to the level of a subcompact cardinal.

We also want to briefly mention that some intuition from the core model induction technique motivates us to consider determinacy models that are not of the form $V=L(\power(\mathbb{R}))$.
Core model induction is the inner model theoretic technique used to obtain models of strong determinacy axioms of the form $V=L(\power(\mathbb{R}))$ from various natural assumptions such as the Proper Forcing Axiom ($\mathsf{PFA}$).
The best result on the lower bound of the consistency strength of $\mathsf{PFA}$ is obtained by this technique:
Trang and the third author showed in \cite{LSAbook} that $\mathsf{PFA}$ implies that there is a model of the Largest Suslin Axiom ($\mathsf{LSA}$), which is much stronger than $\Theta\mathsf{reg}$ in terms of consistency strength but still weaker than a Woodin limit of Woodin cardinals.
Although we expect that their result will be extended to reach, at least, the level of a Woodin limit of Woodin cardinals, in \cite{SaTr}, Trang and the third author showed that the current framework of the core model induction will never reach that level.
They suspect that future core model induction arguments will have to produce determinacy models that are not of the form $V=L(\power(\mathbb{R}))$ to overcome this difficulty.
See the introduction of \cite{SaTr} for further discussion.

\subsection{Beyond Woodin's derived model theorem}

There is a canonical construction of models of $\mathsf{AD}^+ + V=L(\power(\mathbb{R}))$ from large cardinals due to Woodin.
Let $\delta$ be a limit of Woodin cardinals of $V$ and let $g\subset\mathrm{Col}(\omega, {<}\delta)$.
Then the \emph{derived model at $\delta$} (computed in $V[g]$), denoted by $\mathsf{DM}$, is defined as follows:
Let $\mathbb{R}^*_g = \bigcup_{\alpha<\delta} \mathbb{R}^{V[g\uphar\alpha]}$, where $g\uphar\alpha:=g\cap\mathrm{Col}(\omega, {<}\alpha)$.
Let
\[
\Gamma^*_g = \{A^*_g\subset\mathbb{R}^*_g\mid\exists\alpha<\delta( A\subset\mathbb{R}^{V[g\uphar\alpha]}\land V[g\uphar\alpha]\models\text{$A$ is ${<}\delta$-universally Baire})\}.
\]
Here, we write $A^*_g = \bigcup_{\beta\in(\alpha, \delta)} A^{g\uphar\beta}$, where $A^{g\uphar\beta}$ is the canonical extension of $A$ in $V[g\uphar\beta]$ via its ${<}\delta$-universally Baire representation.
Then $\mathsf{DM}=L(\Gamma^*_g, \mathbb{R}^*_g)$.\footnote{$\mathsf{DM}$ depends on the choice of generic $g$, but its theory does not by the homogeneity of $\mathrm{Col}(\omega, {<}\delta)$. So we sometimes say that ``the'' derived model satisfies a statement.}
Woodin showed that $\mathsf{DM}\models\mathsf{AD}^+$ always holds and that the $\mathsf{DM}$ can satisfy stronger forms of determinacy:
\begin{itemize}
    \item If $\delta$ is also a limit of ${<}\delta$-strong cardinals, then $\mathsf{DM}\models\mathsf{AD}_{\mathbb{R}}$.
    \item If there is a cardinal $\kappa<\delta$ that is $\delta$-supercompact, then $\mathsf{DM}\models\Theta\mathsf{reg}$.\footnote{The proof of this result is not written up anywhere to the best of our knowledge.}
\end{itemize}
For basic properties of derived models, see \cite{St09}.

There are several examples generalizing the derived model construction to produce models that are not of the form $V=L(\power(\mathbb{R}))$:
Woodin showed that Solovay's model $L(\mathbb{R}, \mu)$ of $\mathsf{AD}+$``$\omega_1$ is $\mathbb{R}$-supercompact'' can be realized by a generalized derived model construction (\cite{Tr15}).
This technique was extended by Trang to produce a model $L(\power(\mathbb{R}))[\mu]$ of $\Theta\mathsf{reg}+$``$\omega_1$ is $\power(\mathbb{R})$-supercompact'' in \cite{Tr15MALQ}.
Also, Larson--Sargsyan--Wilson's model of $\mathsf{AD}+$ ``all subsets of reals are universally Baire'' in \cite{LSW} is an example of a generalized derived model. Their model does not even have the form $V=L(A)$ for a set $A$.
Here, we are interested in generalizations of the derived model theorem to different kinds of models: Chang-type models of determinacy in the spirit of the model constructed in \cite{CoveringChang}.
Let $\mathsf{CM}=L({}^{\omega}\!\ord)$ be the Chang model and let $\mathsf{CM}^+ = L({}^{\omega}\!\ord)[\langle\mu_{\alpha}\mid\alpha\in\ord\rangle]$, where $\mu_{\alpha}$ is the club filter on $\power_{\omega_1}({}^{\omega}\alpha)$.

\begin{thm}[Woodin, \cite{Wo21}]\label{det_in_Chang}
    Assume that there is a proper class of Woodin cardinals that are limits of Woodin cardinals. Then
    \begin{enumerate}
        \item $\mathsf{CM}\models\mathsf{AD}^+$, and
        \item $\mathsf{CM}^+\models\mathsf{AD}^+ + \omega_1\text{ is supercompact.}$
    \end{enumerate}
\end{thm}

The natural question is if $\mathsf{CM}$ and $\mathsf{CM}^+$ can satisfy stronger forms of determinacy such as $\Theta\mathsf{reg}$.
By Mitchell's result in \cite{Mi17}, $\mathsf{CM}$ cannot satisfy $\mathsf{AD}_{\mathbb{R}}$.
On the other hand, Ikegami and Trang showed in \cite{spct_of_omega_1} that assuming that $\omega_1$ is supercompact, $\mathsf{AD}^+$ is equivalent to $\mathsf{AD}_{\mathbb{R}}$.\footnote{Ikegami--Trang's result can be divided into two parts: (i) Assuming $\mathsf{ZF}+\omega_1$ is supercompact, $\mathsf{DC}$ holds. (ii) Assuming $\mathsf{ZF}+\mathsf{DC}+\omega_1$ is $\wp(\mathbb{R})$-strongly compact, $\mathsf{AD}^+$ is equivalent to $\mathsf{AD}_{\mathbb{R}}$.}
So in the setting of \cref{det_in_Chang}, $\mathsf{CM}^+$ is indeed a model of $\mathsf{AD}_{\mathbb{R}}$.
Moreover, Woodin showed that assuming not only a proper class of a Woodin limit of Woodin cardinals but also determinacy of some definable game of length $\omega_1$, $\mathsf{CM}^+$ satisfies that $\Theta$ is regular.
However, the assumption he used is still unknown to be consistent from large cardinals.
We conjecture that the following generalized derived model theorem holds.

\begin{conj}\label{conj_on_global_ver_of_CDM}
    Suppose that $\delta$ is a Woodin cardinal that is a limit of Woodin cardinals. Let $g\subset\mathrm{Col}(\omega, {<}\delta)$ be $V$-generic and let $L(\Gamma^*_g, \mathbb{R}^*_g)$ be the derived model at $\delta$ computed in $V[g]$. Then the following hold in $V(\mathbb{R}^*_g)$:
    \begin{enumerate}
        \item $L({}^{\omega}\!\ord, \Gamma^*_g, \mathbb{R}^*_g)\models\Theta\mathsf{reg}$, and
        \item $L({}^{\omega}\!\ord, \Gamma^*_g, \mathbb{R}^*_g)[\langle\mu_{\alpha}\mid\alpha\in\ord\rangle]\models\Theta\mathsf{reg}+ \omega_1\text{ is supercompact}$,
    \end{enumerate}
    where $\mu_{\alpha}$ is the club filter on $\power_{\omega_1}({}^{\omega}\alpha)$.
\end{conj}

In \cite{CoveringChang}, the third author introduced a new construction of a determinacy model, called the Chang model over the derived model ($\mathsf{CDM}$), inside a symmetric collapse of a hod mouse with infinitely many Woodin cardinals.
This model extends the derived model of the hod mouse by adding all bounded $\omega$-sequences of some ordinal without increasing its set of reals.
So the main result of \cite{CoveringChang} shows that some weaker form of (1) in \cref{conj_on_global_ver_of_CDM} is true in a hod mouse, as witnessed by $\mathsf{CDM}$.
One can found some applications of $\mathsf{CDM}$ in \cite{omega_strongly_meas_in_Pmax,Det_in_Chang_model}.
In this paper, we verify a weaker form of (2) in a hod mouse by constructing a model called the Chang model over the derived model with supercompact measures ($\mathsf{CDM}^+$).
Compared to $\mathsf{CM}$ and $\mathsf{CM}^+$, the advantage of generalized derived models in \cref{conj_on_global_ver_of_CDM} might be that one could prove that they satisfy $\Theta\mathsf{reg}$ in the same way as for the derived model.
This is indeed the case for $\mathsf{CDM}$ and $\mathsf{CDM}^+$.

\subsection{Determinacy and supercompactness of $\omega_1$}

Apart from the potential applications mentioned above, the study of models of determinacy with supercompact measures for $\omega_1$ is interesting in its own right.
This line of research was initiated by Solovay \cite{So21}, who showed that $\mathsf{AD}_{\mathbb{R}}$ implies that $\omega_1$ is $\mathbb{R}$-supercompact.
The existence and uniqueness of supercompact measures for $\omega_1$ under $\mathsf{AD}$ were studied by Harrington--Kechris \cite{HaKe81}, Becker \cite{Be81}, and Woodin \cite{Wo20}, using purely descriptive set theoretic methods.
Woodin and Neeman \cite{Neeman07} extended these results using inner model theory, proving that under $\mathsf{AD} + V = L(\mathbb{R})$, for any $\alpha < \Theta$, there is a unique supercompact measure on $\wp_{\omega_1}(\alpha)$.
Their proof in \cite{Neeman07} relies on the direct limit system of mice to represent HOD up to $\Theta$ in $L(\mathbb{R})$ (cf.\ \cite{StW16}),
Woodin observed that $\mathsf{AD}^+$ suffices for this result, using a relativized direct limit system (cf.\ \cite{St16, Sar21}).
Notably, our construction of determinacy models in this paper also employs a direct limit system of hod mice and shares technical similarities with the work of Woodin and Neeman.

Moreover, many theorems are known regarding the consistency strength of supercompact measures on $\omega_1$ in models of determinacy.
Woodin showed that the theory ``$\mathsf{AD} + \, \omega_1$ is $\mathbb{R}$-supercompact'' is equiconsistent with the existence of $\omega^2$ Woodin cardinals. This result was published for the first time by Trang in \cite{Tr13, Tr15}, who also obtained interesting generalizations of Woodin's result to $\omega^\alpha$ Woodin cardinals for $\alpha < \omega_1$.
Trang showed in \cite{Tr13,Tr15MALQ,EquiSpctMeas} that the theory ``$\Theta\mathsf{reg} + \omega_1 \text{ is } \power(\mathbb{R})\text{-supercompact}$'' is equiconsistent with the theory ``$\mathsf{AD}_{\mathbb{R}} + \; \Theta$ is measurable.''
In \cite{TrWi21}, Trang and Wilson studied strong compactness of $\omega_1$ and they also showed that if $\mathsf{DC}+\omega_1$ is $\power(\mathbb{R})$-supercompact holds, then there is a sharp for a model of $\mathsf{AD}_{\mathbb{R}} + \mathsf{DC}$.
In \cite{Nairian}, Blue, Larson, and the third author proved that the same type of determinacy models used in \cref{BlueLarsonSargsyan} can satisfy that $\omega_1$ is supercompact, and that the consistency strength of ``$\Theta\mathsf{reg}\, + \, \omega_1 \text{ is supercompact}$'' is strictly weaker than the existence of a Woodin cardinal that is a limit of Woodin cardinals.

\subsection{Summary of our main result}

All necessary terminology and notations will be defined in the next section, but we summarize our result here:

\begin{thm}\label{summary_of_main_results}
Let $(\mathcal{V}, \Omega)$ be an excellent least branch hod pair such that $\mathcal{V}\models\mathsf{ZFC}$.
Suppose that in $\mathcal{V}$, $\delta$ is a cardinal that is a limit of Woodin cardinals and, if $\delta$ is not regular, then its cofinality is not measurable.
We let $\mathcal{P}=\mathcal{V}\vert(\delta^+)^{\mathcal{V}}$ and let $\Sigma$ be the $(\omega, \delta+1)$-iteration strategy for $\mathcal{P}$ determined by $\Omega$.
Also, let $g\subset\mathrm{Col}(\omega, {<}\delta)$ be $\mathcal{V}$-generic.
Then there are $\mathcal{Q}\in I^*_g(\mathcal{P}, \Sigma)$ and $\eta<\delta$ such that
\[
\mathsf{CDM}^{+}(\mathcal{Q}, \eta)\models\mathsf{ZF} + \mathsf{AD}^{+} + \mathsf{AD}_{\mathbb{R}} + \text{$\omega_1$ is ${<}\delta^{\mathcal{Q}, \eta}_{\infty}$-supercompact}
\]
and $\delta^{\mathcal{Q}, \eta}_{\infty}$ is a cardinal $\geq\Theta$ in $\mathsf{CDM}^+(\mathcal{Q}, \eta)$.
Moreover,
\begin{itemize}
\item If $\delta$ is regular in $\mathcal{V}$, then $\mathsf{CDM}^{+}(\mathcal{Q}, \eta)\models\mathsf{DC}+\Theta$ is regular $+\,\delta^{\mathcal{Q}, \eta}_{\infty}$ is regular.
\item If $\delta$ is a limit of Woodin cardinals that is also a limit of ${<}\delta$-strong cardinals in $\mathcal{V}$, then in $\mathsf{CDM}^+(\mathcal{Q}, \eta)$, $\delta^{\mathcal{Q}, \eta}_{\infty}>\Theta$, $\Theta$ is measurable, and $\omega_1$ is $\power(\mathbb{R})$-supercompact.
\end{itemize}
\end{thm}

This follows from \cref{main_thm_1,Theta_measurable} as well as \cref{P(R)-supercompactness,main_thm_2,main_thm_3} below.
The third author recently showed in so far unpublished work that the assumption in \cref{summary_of_main_results} is consistent relative to a Woodin limit of Woodin cardinals, but the proof is not published yet.
So the hypothesis in \cref{summary_of_main_results} is weaker than the assumption of Woodin's \cref{det_in_Chang} in terms of consistency strength.
We leave the question on how large $\delta^{\mathcal{Q}, \eta}_{\infty}$ can be for future work but conjecture the following.

\begin{conj}\label{conj_on_value_of_delta_infty}
Suppose that $\delta$ is a Woodin limit of Woodin cardinals in $\mathcal{V}$ and that the conclusion of \cref{summary_of_main_results} holds for $\mathcal{Q}$ and $\eta$.
Then $\delta^{\mathcal{Q}, \eta}_{\infty}$ is a weakly inaccessible cardinal above $\Theta$ in $\mathsf{CDM}^+(\mathcal{Q}, \eta)$.
\end{conj}

Finally, we would like to mention that Steel independently found a variant of $\mathsf{CDM}^+$ starting from a hod mouse $\mathcal{V}$ with a measurable Woodin cardinal $\delta$ (an hypothesis that is currently not known to be consistent).
His model is also constructed in $\mathcal{V}[g]$, where $g\subset\mathrm{Col}(\omega, {<}\delta)$ is $\mathcal{V}$-generic, but unlike our model, it has supercompact measures on $\power_{\omega_1}({}^{\omega}\alpha)$ for all $\alpha<\omega_2^{\mathcal{V}[g]}$.
As an application of this model, Steel extended the first and third authors' work in \cite{Det_in_Chang_model} to show that assuming the existence of a hod mouse with a measurable Woodin cardinal, $\mathsf{CM}^+\models\mathsf{AD}^+ + \omega_1\text{ is supercompact}$.
See \cite{gsnote} for the details.

\subsection*{Acknowledgments}

The authors are grateful to the anonymous referee for numerous helpful comments, which have significantly improved the paper.
This research was funded in whole or in part by the Austrian Science Fund (FWF) [10.55776/V844, 10.55776/Y1498, 10.55776/I6087].
The third author's work is funded by the National Science Centre, Poland under the Weave-UNISONO call in the Weave programme, registration number UMO-2021/03/Y/ST1/00281.

\section{Construction}

We start with recalling basic notions.
For any set $X$, let $\power_{\omega_1}(X)$ be the set of all countable subsets of $X$.
For $C\subset\power_{\omega_1}(X)$, we say
\begin{enumerate}
    \item $C$ is \emph{unbounded} if for any $\sigma\in\power_{\omega_1}(X)$, there is $\tau\in C$ such that $\sigma\subset\tau$.
    \item $C$ is \emph{closed} if whenever $\langle\sigma_n\mid n<\omega\rangle$ is a $\subset$-increasing sequence such that $\sigma_n\in C$ for all $n<\omega$, then $\bigcup_{n<\omega}\sigma_n\in C$.
    \item $C$ is \emph{a club in $\power_{\omega_1}(X)$} \footnote{This notion is sometimes called a weak club.} if $C$ is unbounded and closed.
\end{enumerate}
The \emph{club filter on $\power_{\omega_1}(X)$} is defined as the filter generated by club subsets of $\power_{\omega_1}(X)$.
For a filter $\mu$ on $\power_{\omega_1}(X)$, we say
\begin{enumerate}
    \item $\mu$ is \emph{countably complete} if it is closed under countable intersections.
    \item $\mu$ is \emph{fine} if for any $x\in X$, $\{\sigma\in\power_{\omega_1}(X)\mid x\in\sigma\}\in\mu$.
    \item $\mu$ is \emph{normal} if it is closed under diagonal intersections, i.e., whenever $\langle A_x\mid x\in X\rangle$ is a sequence such that $A_x\in\mu$ for all $x\in X$, then $\triangle_{x\in X} A_x := \{\sigma\in\power_{\omega_1}(X)\mid\sigma\in\bigcap_{x\in\sigma}A_x\}\in\mu$.
\end{enumerate}
For any uncountable set $X$, the club filter on $\power_{\omega_1}(X)$ has all these properties.

\begin{dfn}
Let $X$ be an uncountable set.
A \emph{supercompact measure on $\power_{\omega_1}(X)$} is a countably complete normal fine ultrafilter on $\power_{\omega_1}(X)$.
We say \emph{$\omega_1$ is $X$-supercompact} if there is a supercompact measure on $\power_{\omega_1}(X)$.
Also, we say \emph{$\omega_1$ is supercompact} if $\omega_1$ is $X$-supercompact for any uncountable set $X$.
\end{dfn}

Of course, this definition is meaningful only in the absence of the Axiom of Choice.
See \cite{spct_of_omega_1} for several conclusions from supercompactness of $\omega_1$.

\subsection{Setup}

Our construction of a determinacy model is done inside a symmetric collapse of some hod mouse.
Roughly speaking, a hod premouse is a structure of the form $L_{\alpha}[\vec{E}, \Sigma]$, where $\vec{E}$ is a coherent sequence of extenders and $\Sigma$ is a fragment of its own iteration strategy.
\footnote{A hod premouse is designed for representing $\mathrm{HOD}$ of a determinacy model of the form $L(\power(\mathbb{R}))$, which is why the name includes ``hod.''}
A hod pair is a pair of a hod premouse and its iteration strategy, assuming that this iteration strategy has certain regularity properties.
In this paper, we use Steel's least branch (lbr) hod premice introduced in \cite{CPMP}.
See \cite[Definition 9.2.2]{CPMP} for the precise definition of a hod pair.

To avoid including $\mathsf{AD}^+$ in our background theory, we need to assume regularity properties of the iteration strategy in a hod pair that follow from $\mathsf{AD}^+$.
According to \cite{CoveringChang}, we say that a hod pair $(\mathcal{V}, \Omega)$ is \emph{excellent} if $\mathcal{V}$ is countable, $\Omega$ is $(\omega_1, \omega_1+1)$-iteration strategy for $\mathcal{V}$, and whenever $\mathcal{P}\trianglelefteq\mathcal{V}, \mathcal{P}\cap\ord$ is an inaccessible cardinal of $\mathcal{V}$, $\rho(\mathcal{V})>\mathcal{P}\cap\ord$, and $\Sigma=\Omega_{\mathcal{P}}$, then the following hold:
\begin{enumerate}
    \item $\Sigma$ admits full normalization, i.e., whenever $\mathcal{T}$ is an iteration tree on $\mathcal{P}$ via $\Sigma$ with last model $\mathcal{Q}$, there is a normal iteration $\mathcal{U}$ on $\mathcal{P}$ via $\Sigma$ with last model $\mathcal{Q}$ such that $\pi^{\mathcal{T}}$ exists if and only if $\pi^{\mathcal{U}}$ exists, and if $\pi^{\mathcal{T}}$ exists then $\pi^{\mathcal{T}}=\pi^{\mathcal{U}}$,
    \item $\Sigma$ is positional, i.e., if $\mathcal{Q}$ is a $\Sigma$-iterate of $\mathcal{P}$ via an iteration tree $\mathcal{T}$ and it is also via another iteration tree $\mathcal{U}$, then $\Sigma_{\mathcal{T}, \mathcal{Q}}=\Sigma_{\mathcal{U}, \mathcal{Q}}$,
    \footnote{We then are allowed to denote the unique tail strategy for $\mathcal{Q}$ by $\Sigma_{\mathcal{Q}}$.}
    \item $\Sigma$ is directed, i.e., if $\mathcal{Q}_0$ and $\mathcal{Q}_1$ are $\Sigma$-iterates of $\mathcal{P}$ via iteration trees above some ordinal $\eta$, then there is an $\mathcal{R}$ such that $\mathcal{R}$ is a $\Sigma_{\mathcal{Q}_i}$-iterate of $\mathcal{Q}_i$ via an iteration tree above $\eta$ for any $i\in\{0, 1\}$,
    \item $(\mathcal{P}, \Sigma)$ satisfies generic interpretability in the sense of \cite[Theorem 11.1.1]{CPMP}, and
    \item $\Sigma$ is segmentally normal, i.e., whenever $\eta$ is inaccessible cardinal of $\mathcal{P}$ such that $\rho(\mathcal{P})>\eta$, $\mathcal{Q}$ is a non-dropping $\Sigma$-iterate of $\mathcal{P}$ via an iteration tree $\mathcal{T}$ that is above $\eta$, and $\mathcal{R}$ is a non-dropping $\Sigma_{\mathcal{Q}}$-iterate of $\mathcal{Q}$ via an iteration tree $\mathcal{U}$ that is based on $\mathcal{Q}\vert\eta$, then $\Sigma_{\mathcal{P}\vert\eta}=(\Sigma_{\mathcal{Q}})_{\mathcal{P}\vert\eta}$ and letting $\mathcal{R}^*$ be a non-dropping $\Sigma$-iterate of $\mathcal{P}$ via the iteration tree $\mathcal{U}^*$ that has the same extenders and branches as $\mathcal{U}$, $\mathcal{R}$ is a non-dropping $\Sigma_{\mathcal{R}^*}$-iterate of $\mathcal{R}^*$ via a normal iteration tree that is above $\pi_{\mathcal{P}, \mathcal{R}^*}(\eta)$.
\end{enumerate}
Siskind and Steel showed that under $\mathsf{AD}^+$, every countable hod pair is excellent (\cite{CPMP, full_normalization}).
Our definition of excellence has slight differences from \cite[Definition 2.1]{CoveringChang}.
First, we omit stability and pullback consistency from the definition because they are already part of the definition of a hod pair in \cite{CPMP}.
Also, we do not restrict to strongly non-dropping iteration trees, simply because it turns out that we do not have to.
See the remark after \cref{genericity_iteration} as well.
The consequence of excellence that the reader should be particularly aware of is that if a hod pair $(\mathcal{V}, \Omega)$ is excellent, then 
\begin{itemize}
    \item for any $\mathcal{P}$ and $\Sigma$ as in the definition of excellence, $\Sigma$ has a canonical extension $\Sigma^g$ in $\mathcal{P}[g]$, where $g\subset\mathrm{Col}(\omega, {<}\delta)$ is $\mathcal{P}$-generic and $\delta$ is the supremum of all Woodin cardinals of $\mathcal{P}$, and
    \item internal direct limit models as defined in \cref{dfn:internal_direct_limit_system} are well-defined.
\end{itemize}

Now we describe our setup, which is the same as in \cite{CoveringChang}.
Let $(\mathcal{V}, \Omega)$ be an excellent hod pair such that $\mathcal{V}\models\mathsf{ZFC}$.
Suppose that in $\mathcal{V}$, $\delta$ is a cardinal that is a limit of Woodin cardinals and if $\delta$ is not regular, then its cofinality is not measurable.
\footnote{Throughout this paper, we adopt the following standard convention: if $\mathcal{M}$ is an lbr hod premouse, then ``$\delta$ has some large cardinal property in $\mathcal{M}$'' actually means ``the extender sequence of $\mathcal{M}$ witnesses that $\delta$ has some large cardinal property in $\mathcal{M}$.''}
We let $\mathcal{P}=\mathcal{V}\vert(\delta^+)^{\mathcal{V}}$ and let $\Sigma$ be the $(\omega, \delta+1)$-iteration strategy for $\mathcal{P}$ determined by the strategy predicate of $\mathcal{V}$.
Also, let $g\subset\mathrm{Col}(\omega, {<}\delta)$ be $\mathcal{V}$-generic.
We fix the objects defined in this paragraph throughout the paper and work in $\mathcal{V}[g]$ unless otherwise noted.

Let $\mathsf{DM}=L(\Gamma^*_g, \mathbb{R}^*_g)$ be the derived model at $\delta$ computed in $\mathcal{V}[g]$.
The following result is part of our motivation to study the Chang model over the derived model together with $\Theta\mathsf{reg}$.

\begin{thm}[\cite{CPMP,dm_of_self_it}]\label{determinacy_in_DM}
The set of all sets of reals in $\mathsf{DM}$ is $\Gamma^*_g$ and $\mathsf{DM}\models\mathsf{AD}^{+}+\mathsf{AD}_{\mathbb{R}}$.
Moreover, if $\delta$ is regular in $\mathcal{V}$, then $\mathsf{DM}\models\Theta$ is regular.
\end{thm}
\begin{proof}
Steel showed the first part as \cite[Theorem 11.3.2]{CPMP}.
In \cite{dm_of_self_it}, the first and the third author generalized his result to any self-iterable structure, and additionally showed that the derived model of a self-iterable structure at a regular limit of Woodin cardinals satisfies that $\Theta$ is regular.
\end{proof}

Now we proceed with defining the Chang model over derived model introduced by the third author in \cite{CoveringChang}.
To state the definition, we need more terminology and facts.
We define $I^*_g(\mathcal{P}, \Sigma)$ as the set of all non-dropping \footnote{We say that $\mathcal{Q}$ is a non-dropping iterate of $\mathcal{P}$ via $\mathcal{T}$ if the main branch of $\mathcal{T}$ does not drop.} $\Sigma$-iterates of $\mathcal{P}$ via an $(\omega, \delta+1)$-iteration tree $\mathcal{T}$ of $\mathcal{P}$ based on $\mathcal{P}\vert\delta$ \footnote{For an iteration tree $\mathcal{T}$ on $\mathcal{P}$, we say that $\mathcal{T}$ is based on $\mathcal{P}\vert\delta$ if it only uses extenders on the extender sequence of $\mathcal{P}\vert\delta$ and their images.} such that $\pi^{\mathcal{T}}(\delta)=\delta$ and $\mathcal{T}\in\mathcal{V}[g\uphar\xi]$ for some $\xi<\delta$.

Let $\mathcal{Q}\in I^*_g(\mathcal{P}, \Sigma)$.
Because $\Sigma$ (and its canonical extensions to generic extensions) admits full normalization, $\mathcal{Q}$ is a non-dropping normal $\Sigma$-iterate of $\mathcal{P}$.
So, let $\mathcal{T}_{\mathcal{P}, \mathcal{Q}}$ be a unique normal iteration tree of $\mathcal{P}$ via $\Sigma$ with last model $\mathcal{Q}$.
Note that the length of $\mathcal{T}_{\mathcal{P}, \mathcal{Q}}$ is at most $\delta+1$.
Let $\Sigma_{\mathcal{Q}}$ be the tail strategy $\Sigma_{\mathcal{Q}, \mathcal{T}_{\mathcal{P}, \mathcal{Q}}}$.
Since $\Sigma$ is positional, $\Sigma_{\mathcal{Q}}=\Sigma_{\mathcal{Q}, \mathcal{U}}$ for any $\Sigma$-iteration $\mathcal{U}$ of $\mathcal{P}$ leading to $\mathcal{Q}$.
Let $\pi_{\mathcal{P}, \mathcal{Q}}\colon\mathcal{P}\to\mathcal{Q}$ be the iteration map via $\mathcal{T}_{\mathcal{P}, \mathcal{Q}}$.
Moreover, since $\mathcal{V}$ does not project across $(\delta^+)^{\mathcal{V}}$, we can apply $\mathcal{T}_{\mathcal{P}, \mathcal{Q}}$ to $\mathcal{V}$ according to $\Omega$.
Then let $\mathcal{V}_{\mathcal{Q}}$ be the last model of $\mathcal{T}_{\mathcal{P}, \mathcal{Q}}$ when it is applied to $\mathcal{V}$.
It is not hard to check that $\mathcal{Q}=\mathcal{V}_{\mathcal{Q}}\vert(\delta^+)^{\mathcal{V}_{\mathcal{Q}}}$ and $\Sigma_{\mathcal{Q}}$ is determined by the strategy predicate of $\mathcal{V}_{\mathcal{Q}}$.

\begin{dfn}\label{dfn:internal_direct_limit_system}
For any $\mathcal{Q}\in I^*_{g}(\mathcal{P}, \Sigma)$ and any ordinal $\eta<\delta$\footnote{Recall that $\delta=\pi_{\mathcal{P}, \mathcal{Q}}(\delta)$ for any $\mathcal{Q}\in I^*_{g}(\mathcal{P}, \Sigma)$.}, we define
\[
\mathcal{F}^*_g(\mathcal{Q}, \eta)
\]
as the set of all non-dropping $\Sigma_{\mathcal{Q}}$-iterates $\mathcal{R}$ of $\mathcal{Q}$ such that $\lh(\mathcal{T}_{\mathcal{Q}, \mathcal{R}})<\delta$, $\mathcal{T}_{\mathcal{Q}, \mathcal{R}}$ is based on $\mathcal{P}\vert\delta$ and is above $\eta$\footnote{We say that an iteration tree is above $\eta$ if it uses only extenders with critical point $>\eta$.}, and $\mathcal{T}_{\mathcal{Q}, \mathcal{R}}\in\mathcal{V}[g\uphar\xi]$ for some $\xi<\delta$.
Since $\Sigma$ is directed, $\mathcal{F}^*_g(\mathcal{Q}, \eta)$ can be regarded as a direct limit system under iteration maps.
We also define
\[
\mathcal{M}_{\infty}(\mathcal{Q}, \eta)
\]
as the direct limit model of the system $\mathcal{F}^*_g(\mathcal{Q}, \eta)$.
For any $\mathcal{R}\in\mathcal{F}^*_g(\mathcal{Q}, \eta)$, let $\pi^{\mathcal{Q}, \eta}_{\mathcal{R}, \infty}\colon\mathcal{R}\to\mathcal{M}_{\infty}(\mathcal{Q}, \eta)$ be the direct limit map.
Let $\delta^{\mathcal{Q}, \eta}_{\infty}=\pi^{\mathcal{Q}, \eta}_{\mathcal{Q}, \infty}(\delta)$.
\end{dfn}

Let $\mathcal{Q}$ and $\eta$ be as in \cref{dfn:internal_direct_limit_system}.
Since any iteration tree based on $\mathcal{Q}\vert\delta$ can be applied to $\mathcal{V}_{\mathcal{Q}}$, we can similarly define a direct limit system $\mathcal{F}^*_g(\mathcal{V}_{\mathcal{Q}}, \eta)$, which consists of models $\mathcal{V}_{\mathcal{R}}$ and iteration maps $\pi_{\mathcal{V}_{\mathcal{R}}, \mathcal{V}_{\mathcal{R}^*}}$, where $\mathcal{R}, \mathcal{R}^*\in\mathcal{F}^*_g(\mathcal{Q}, \eta)$ are such that $\mathcal{R}^*$ is a non-dropping iterate of $\mathcal{R}$.
It is not hard to see that $\mathcal{V}_{\mathcal{M}_{\infty}(\mathcal{Q}, \eta)}$ is the direct limit model of $\mathcal{F}^*_g(\mathcal{V}_{\mathcal{Q}}, \eta)$.
For any $\mathcal{R}\in\mathcal{F}^*_g(\mathcal{Q}, \eta)$, let $\pi^{\mathcal{Q}, \eta}_{\mathcal{V}_{\mathcal{R}}, \infty}\colon\mathcal{V}_{\mathcal{R}}\to\mathcal{V}_{\mathcal{M}_{\infty}(\mathcal{Q}, \eta)}$ be the corresponding direct limit map which extends $\pi^{\mathcal{Q}, \eta}_{\mathcal{R}, \infty}\colon\mathcal{R}\to\mathcal{M}_{\infty}(\mathcal{Q}, \eta)$.

In \cite{CoveringChang}, the Chang model over the derived model (at $\delta$ computed in $\mathcal{V}[g]$) is defined by\footnote{In \cite{CoveringChang}, this model is denoted by $C(g)$.}
\[
\mathsf{CDM}=L(\mathcal{M}_{\infty}(\mathcal{P}, 0), \cup_{\alpha<\delta^{\mathcal{P}, 0}_{\infty}}{}^{\omega}\alpha, \Gamma^*_g, \mathbb{R}^*_g)
\]
and it is proved that $\mathsf{CDM}\models\mathsf{AD}^+$.
The main object we study in this paper is an extension of $\mathsf{CDM}$ and is introduced in the following definition.

\begin{dfn}\label{def_of_CDM_plus}
Let $\mathcal{Q}\in I^*_{g}(\mathcal{P}, \Sigma)$ and let $\eta<\delta$ be an ordinal.
We define
\[
\mathsf{CDM}(\mathcal{Q}, \eta)=L(\mathcal{M}_{\infty}(\mathcal{Q}, \eta), \cup_{\alpha<\delta^{\mathcal{Q}, \eta}_{\infty}}{}^{\omega}\alpha, \Gamma^*_g, \mathbb{R}^*_g).
\]
Moreover, let $\mu_{\alpha}$ be the club filter on $\power_{\omega_1}(\alpha)$ for any $\alpha<\delta^{\mathcal{Q}, \eta}_{\infty}$ and then let $\vec{\mu}=\{\langle\alpha, A\rangle\mid\alpha<\delta^{\mathcal{Q}, \eta}_{\infty}\land A\in\mu_{\alpha}\}$.
We define 
\[
\mathsf{CDM}^{+}(\mathcal{Q}, \eta)=L(\mathcal{M}_{\infty}(\mathcal{Q}, \eta), \cup_{\alpha<\delta^{\mathcal{Q}, \eta}_{\infty}}{}^{\omega}\alpha, \Gamma^*_g, \mathbb{R}^*_g)[\vec{\mu}].
\]
For any ordinal $\gamma$, we write $\mathsf{CDM}^{+}(\mathcal{Q}, \eta)\vert\gamma$ for the $\gamma$-th level of the $L$-hierarchy of $\mathsf{CDM}^+(\mathcal{Q}, \eta)$.
More precisely, for any ordinal $\gamma$,
\begin{align*}
    \mathsf{CDM}^+(\mathcal{Q}, \eta)\vert 0 &= \mathrm{trcl}(\{\mathcal{M}_{\infty}(\mathcal{Q}, \eta), \cup_{\alpha<\delta^{\mathcal{Q}, \eta}_{\infty}}{}^{\omega}\alpha, \Gamma^*_g, \mathbb{R}^*_g\}),\\
    \mathsf{CDM}^+(\mathcal{Q}, \eta)\vert\gamma+1 &= \mathrm{Def}(\mathsf{CDM}^+(\mathcal{Q}, \eta)\vert\gamma, {\in}, \vec{\mu}\cap\mathsf{CDM}^+(\mathcal{Q}, \eta)\vert\gamma),\\
    \mathsf{CDM}^+(\mathcal{Q}, \eta)\vert\gamma &= \bigcup_{\beta<\gamma}\mathsf{CDM}^+(\mathcal{Q}, \eta)\vert\beta\;\;\;\text{if $\gamma$ is limit},
\end{align*}
where $\mathrm{trcl}(A)$ denotes the transitive closure of $A$ and $\mathrm{Def}$ denotes the definable powerset operator.
Also, we define the language for $\mathsf{CDM}^+(\mathcal{Q}, \eta)$ as the language of set theory together with an additional unary predicate $\dot{\vec{\mu}}$.
$\mathsf{CDM}^+(\mathcal{Q}, \eta)\vert\gamma$ always interprets $\dot{\vec{\mu}}$ as $\vec{\mu}\cap\mathsf{CDM}^+(\mathcal{Q}, \eta)\vert\gamma$.
\end{dfn}

Note that we add the club filters on $\power_{\omega_1}(\xi)$, not on $\power_{\omega_1}({}^{\omega}\xi)$, which is different from Woodin's generalized Chang model $\mathsf{CM}^+$.
The reason for this will be explained in \cref{rem_on_def_of_CDM_plus}.
Our goal in this section is to show the following.

\begin{thm}\label{main_thm_1}
There are $\mathcal{Q}\in I^*_g(\mathcal{P}, \Sigma)$ and $\eta<\delta$ such that
\[
\mathsf{CDM}^{+}(\mathcal{Q}, \eta)\models\mathsf{ZF} + \mathsf{AD}^{+} + \mathsf{AD}_{\mathbb{R}} + \text{$\omega_1$ is ${<}\delta^{\mathcal{Q}, \eta}_{\infty}$-supercompact.}
\]
Moreover,
\begin{itemize}
\item if $\delta$ is regular in $\mathcal{V}$, then $\mathsf{CDM}^{+}(\mathcal{Q}, \eta)\models\mathsf{DC}+\Theta$ is regular.
\item if $\delta^{\mathcal{Q}, \eta}_{\infty}>\Theta$ in $\mathsf{CDM}^+(\mathcal{Q}, \eta)$, then
\[
\mathsf{CDM}^+(\mathcal{Q}, \eta)\models
\Theta\text{ is measurable}+\omega_1\text{ is } \power(\mathbb{R})\text{-supercompact.}
\]
\end{itemize}
\end{thm}

We will prove ${<}\delta^{\mathcal{Q}, \eta}_{\infty}$-supercompactness of $\omega_1$ as \cref{spct_meas_in_CDM_plus} in Subsection 2.3, $\mathsf{AD}^+ +\mathsf{AD}_{\mathbb{R}}$ as \cref{determinacy_in_CDM_plus_1} in Subsection 2.4, 
$\mathsf{DC}$ and regularity of $\Theta$ as \cref{determinacy_in_CDM_plus_2} in Subsection 2.5, and measurability of $\Theta$ and $\power(\mathbb{R})$-supercompactness of $\omega_1$ as \cref{Theta_measurable,P(R)-supercompactness} in Subsection 2.6.

\subsection{Genericity iterations}

We first need to introduce genericity iterations in our context and recall several lemmas proved in \cite{CoveringChang}.
Let $\mathcal{M}$ be an lbr hod premouse.
Then we say that an open interval of ordinals $(\eta, \delta)$ is a \emph{window of $\mathcal{M}$} if in $\mathcal{M}$, $\eta$ is an inaccessible cardinal and $\delta$ is the least Woodin cardinal above $\eta$ in $\mathcal{M}$.
For any iteration tree $\mathcal{T}$ on $\mathcal{M}$, we say that \emph{$\mathcal{T}$ is based on a window $(\eta, \delta)$} if it is based on $\mathcal{M}\vert\delta$ and is above $\eta$, i.e., $\mathcal{T}$ uses only extenders on the extender sequence of $\mathcal{M}\vert\delta$ with critical point $>\eta$ and their images.
Also, a sequence $\langle w_{\alpha}\mid\alpha<\lambda\rangle$ of windows of $\mathcal{M}$ is \emph{increasing} if whenever $\alpha<\beta$, $\sup(w_{\alpha})\leq\inf(w_{\beta})$.

\begin{dfn}
Let $\mathcal{Q}\in I^*_g(\mathcal{P}, \Sigma)$ and let $\mathcal{R}\in I^*_g(\mathcal{Q}, \Sigma_{\mathcal{Q}})$.
We say that $\mathcal{R}$ is a \emph{window-based iterate of $\mathcal{Q}$} if there is $\xi<\delta$ such that $\mathcal{R}\in\mathcal{V}[g\uphar\xi]$, an increasing sequence of windows $\langle w_{\alpha}\mid\alpha<\cf(\delta)\rangle$ of $\mathcal{R}$ and a sequence $\langle\mathcal{R}_{\alpha}\mid\alpha\leq\cf(\delta)\rangle$ of lbr hod premice in $\mathcal{V}[g\uphar\xi]$ such that
\begin{enumerate}
    \item $\delta=\sup\{\sup(w_{\alpha})\mid\alpha<\cf(\delta)\}$.
    \item $\mathcal{R}_0$ is a non-dropping iterate of $\mathcal{Q}$ based on $\mathcal{Q}\vert\inf(w_0)$.
    \item $\mathcal{R}_{\alpha+1}$ is a non-dropping iterate of $\mathcal{R}_{\alpha}$ based on a window $\pi_{\mathcal{Q}, \mathcal{R}_{\alpha}}(w_{\alpha})$.
    \item for any limit ordinal $\lambda\leq\cf(\delta)$, $\mathcal{R}_{\lambda}$ is the direct limit of $\langle\mathcal{R}_{\alpha}, \pi_{\mathcal{R}_{\alpha}, \mathcal{R}_{\beta}}\mid\alpha<\beta<\lambda\rangle$.
    \item $\mathcal{R}=\mathcal{R}_{\cf(\delta)}$.
\end{enumerate}
\end{dfn}

Let $\mathcal{M}$ be an lbr hod premouse.
An extender $E\in\vec{E}^{\mathcal{M}}$ is called \emph{nice} if the supremum of the generators of $E$ is an inaccessible cardinal in $\mathcal{M}$.
For any window $w=(\eta, \delta)$ of $\mathcal{R}$, let $\mathsf{EA}^{\mathcal{M}}_{(\eta, \delta)}$ be Woodin's extender algebra with $\omega$ generators at $\delta$ in $\mathcal{M}$ that only uses nice extenders $E$ such that $\crit(E)>\eta$, see \cite{Farah} and \cite{OIMT}.

\begin{dfn}\label{genericity_iteration}
Let $\mathcal{Q}\in I^*_g(\mathcal{P}, \Sigma)$ and let $\mathcal{R}\in I^*_g(\mathcal{Q}, \Sigma_{\mathcal{Q}})$.
We say that $\mathcal{R}$ is a \emph{genericity iterate of $\mathcal{Q}$} if it is a window-based iterate of $\mathcal{Q}$ as witnessed by $\langle w_{\alpha}\mid\alpha<\cf(\delta)\rangle$ such that
\begin{enumerate}
    \item for any $x\in\mathbb{R}^{\mathcal{P}[g]}$, there is an $\alpha<\delta$ such that $x$ is $\mathsf{EA}^{\mathcal{R}}_{\pi_{\mathcal{Q}, \mathcal{R}}(w_{\alpha})}$-generic over $\mathcal{R}$, and
    \item for any $\alpha<\cf(\delta)$, $w_{\alpha}\in\mathrm{ran}(\pi_{\mathcal{Q}, \mathcal{R}})$.
\end{enumerate}
We say that $\mathcal{R}$ is a \emph{genericity iterate of $\mathcal{Q}$ above $\eta$} if it is a genericity iterate of $\mathcal{Q}$ witnessed by $\langle w_{\alpha}\mid\alpha<\cf(\delta)\rangle$ and $\langle\mathcal{R}_{\alpha}\mid\alpha\leq\cf(\delta)\rangle$ such that $\inf(w_0)\geq\eta$.
\end{dfn}

In \cite{CoveringChang}, a genericity iteration is required to be strongly non-dropping, or use only nice extenders.
This condition is actually redundant, so we omit it from \cref{genericity_iteration}.
The following lemma is a restatement of \cite[Propositions 3.3 and 3.4]{CoveringChang}.

\begin{lem}\label{easy_lemma}
Let $\eta<\delta$. Then the following hold.
\begin{enumerate}
    \item For any $\mathcal{P}^*\in\mathcal{F}^*_{g}(\mathcal{P}, \eta)$ and any $\eta'<\delta$, there is $\mathcal{Q}\in I^*_{g}(\mathcal{P}^*, \Sigma_{\mathcal{P}^*})$ such that $\mathcal{Q}$ is a genericity iterate of $\mathcal{P}$, $\crit(\pi_{\mathcal{P}^*, \mathcal{Q}})>\eta'$, and $\mathcal{T}_{\mathcal{P}, \mathcal{P}^*}{}^{\frown}\mathcal{T}_{\mathcal{P}^*, \mathcal{Q}}$ is a normal iteration tree.
    \item If $\mathcal{Q}$ is a genericity iterate of $\mathcal{P}$ above $\eta$ and $\mathcal{R}$ is a genericty iterate of $\mathcal{Q}$ above $\eta$, then $\mathcal{R}$ is a genericity iterate of $\mathcal{P}$ above $\eta$.
\end{enumerate}
\end{lem}

The proof of \cite[Theorem 3.8]{CoveringChang} shows the following.

\begin{lem}\label{pres_of_M_infty}
    Let $\mathcal{Q}\in I^*_g(\mathcal{P}, \Sigma)$ and $\eta<\delta$.
    If $\mathcal{R}$ is any genericity iterate of $\mathcal{Q}$ above $\eta$, then
    \[
    \mathcal{M}_{\infty}(\mathcal{Q}, \eta)=\mathcal{M}_{\infty}(\mathcal{R}, \eta).
    \]
    Moreover, $\pi^{\mathcal{Q}, \eta}_{\mathcal{Q}, \infty}=\pi^{\mathcal{R}, \eta}_{\mathcal{R}, \infty}\circ\pi_{\mathcal{Q}, \mathcal{R}}$.
    In particular, $\delta^{\mathcal{Q}, \eta}_{\infty}=\delta^{\mathcal{R}, \eta}_{\infty}$.
\end{lem}

\begin{cor}\label{pres_of_CDM_plus}
    Let $\mathcal{Q}\in I^*_g(\mathcal{P}, \Sigma)$ and $\eta<\delta$.
    If $\mathcal{R}$ is any genericity iterate of $\mathcal{Q}$ above $\eta$, then
    \begin{align*}
        \mathsf{CDM}(\mathcal{Q}, \eta) &= \mathsf{CDM}(\mathcal{R}, \eta), \\
        \mathsf{CDM}^+(\mathcal{Q}, \eta) &= \mathsf{CDM}^+(\mathcal{R}, \eta),
    \end{align*}
    where these models are defined in $\mathcal{V}[g]$.
\end{cor}

Now let $\mathcal{Q}$ be a genericity iterate of $\mathcal{P}$.
Then there is a $\mathcal{Q}$-generic $h\subset\mathrm{Col}(\omega, {<}\delta)$ (in $\mathcal{V}[g]$) such that $(\mathbb{R}^*_g)^{\mathcal{P}[g]}=(\mathbb{R}^*_h)^{\mathcal{Q}[h]}$.
We call such an $h$ \emph{maximal}.
The proof of \cite[Proposition 4.2]{CoveringChang} shows the following.

\begin{lem}\label{locality_of_CDM}
Let $\mathcal{Q}\in I^*_g(\mathcal{P}, \Sigma)$ and $\eta<\delta$.
If $h\subset\mathrm{Col}(\omega, {<}\delta)$ is a maximal $\mathcal{Q}$-generic, then
\[
\mathsf{CDM}(\mathcal{Q}, \eta)=(\mathsf{CDM}(\mathcal{Q}, \eta))^{\mathcal{V}_{\mathcal{Q}}[h]}.
\]
\end{lem}

\subsection{Supercompact measures on $\power_{\omega_1}(\alpha)$}

We would like to generalize \cref{locality_of_CDM} to $\mathsf{CDM}^+(\mathcal{Q}, \eta)$, which is crucial for almost all our proofs of the results in this paper.
It does not seem that this is true for arbitrary $\mathcal{Q}$ and $\eta$, so we need to describe how $\mathcal{Q}$ and $\eta$ in \cref{main_thm_1} should be chosen.
Note that if $\eta\leq\eta'<\delta$ then $\delta^{\mathcal{Q}, \eta}_{\infty}\geq\delta^{\mathcal{Q}, \eta'}_{\infty}$ just because $\mathcal{F}^*_g(\mathcal{Q}, \eta')$ is a subsystem of $\mathcal{F}^*_g(\mathcal{Q}, \eta)$.
In general, $\delta^{\mathcal{Q}, \eta}_{\infty}>\delta^{\mathcal{Q}, \eta'}_{\infty}$ is possible, see, for example, \cref{value_of_Theta} below.
The following lemma is trivial, but it is actually one of the key observations in this paper.

\begin{lem}\label{stabilizing_delta_infty}
    There is a genericity iterate $\mathcal{Q}$ of $\mathcal{P}$ and an ordinal $\eta<\delta$ such that for any genericity iterate $\mathcal{R}$ of $\mathcal{Q}$ above $\eta$ and any ordinal $\xi\in [\eta, \delta)$, $\delta^{\mathcal{Q}, \eta}_{\infty}=\delta^{\mathcal{R}, \xi}_{\infty}$.
\end{lem}
\begin{proof}
    Suppose not.
    By \cref{easy_lemma}(2), one can inductively find $\langle\mathcal{Q}_n, \eta_n\mid n<\omega\rangle$ such that for any $n<\omega$, $\mathcal{Q}_{n+1}$ is a genericity iterate of $\mathcal{Q}_n$, $\eta_n<\eta_{n+1}$, and $\delta^{\mathcal{Q}_n, \eta_n}_{\infty}>\delta^{\mathcal{Q}_{n+1}, \eta_{n+1}}_{\infty}$.
    This is a contradiction as we have found a strictly decreasing infinite sequence of ordinals in $\mathcal{V}[g]$.
\end{proof}

We say that \emph{$(\mathcal{Q}, \eta)$ stabilizes $\delta_{\infty}$} if it satisfies the conclusion of \cref{stabilizing_delta_infty}.

\begin{ques}\label{ques_on_stabilizing_delta_infty}
Does some large cardinal assumption on $\delta$ in $\mathcal{V}$ imply that $(\mathcal{P}, 0)$ stabilizes $\delta_{\infty}$?
\end{ques}

An affirmative answer to \cref{ques_on_stabilizing_delta_infty} might be useful because $\mathcal{M}_{\infty}(\mathcal{P}, 0)$ extends $\mathrm{HOD}$ up to $\Theta$ in $\mathsf{CDM}^+(\mathcal{P}, 0)$.
For the results in this paper we do not need to answer \cref{ques_on_stabilizing_delta_infty} as we will simply work above some fixed $(\mathcal{Q}, \eta)$ that stabilizes $\delta_{\infty}$.
Now we are ready to prove the main lemma.

\begin{lem}\label{locality_of_CDM_plus}
    Let $\mathcal{Q}$ be a genericity iterate of $\mathcal{P}$ and let $\eta<\delta$ be such that $(\mathcal{Q}, \eta)$ stabilizes $\delta_{\infty}$.
    Then, whenever $\mathcal{Q}'\in\mathcal{V}_{\mathcal{Q}}[h]$ is a genericity iterate of $\mathcal{Q}$ above $\eta$ and $h'\subset\mathrm{Col}(\omega, {<}\delta)$ is a maximal $\mathcal{Q}'$-generic such that $h'\in\mathcal{V}_{\mathcal{Q}}[h]$,
    \[
    \mathsf{CDM}^+(\mathcal{Q}', \eta)=(\mathsf{CDM}^+(\mathcal{Q}', \eta))^{\mathcal{V}_{\mathcal{Q}'}[h']}.
    \]
\end{lem}
\begin{proof}
    We show 
    \[
    \mathsf{CDM}^+(\mathcal{Q}', \eta)\vert\gamma=(\mathsf{CDM}^+(\mathcal{Q}', \eta)\vert\gamma)^{\mathcal{V}_{\mathcal{Q}'}[h']}
    \]
    by induction on $\gamma$.
    If $\gamma=0$, then it follows from \cref{locality_of_CDM}.
    As the limit steps are trivial, it is enough to consider the successor steps.
    So suppose that $\mathsf{CDM}^+(\mathcal{Q}', \eta)\vert\gamma=(\mathsf{CDM}^+(\mathcal{Q}', \eta)\vert\gamma)^{\mathcal{V}_{\mathcal{Q}'}[h']}$.
    It suffices to show that for all $\alpha\in\gamma\cap\delta^{\mathcal{Q}', \eta}_{\infty}$,
    \[
    \mu_{\alpha}\cap\mathsf{CDM}^+(\mathcal{Q}', \eta)\vert\gamma=\mu_{\alpha}^{\mathcal{V}_{\mathcal{Q}'}[h']}\cap\mathsf{CDM}^+(\mathcal{Q}', \eta)\vert\gamma,
    \]
    which implies that $\mathsf{CDM}^+(\mathcal{Q}', \eta)\vert\gamma+1=(\mathsf{CDM}^+(\mathcal{Q}', \eta)\vert\gamma+1)^{\mathcal{V}_{\mathcal{Q}'}[h']}$.
    We fix such $\alpha$ for the rest of the proof.
    
    Let $\mathcal{R}$ be a genericity iterate of $\mathcal{Q}'$, let $k\subset\mathrm{Col}(\omega, {<}\delta)$ be a maximal $\mathcal{R}$-generic, and let $\xi<\delta$.
    For any $\mathcal{R}^*\in\mathcal{F}^*_k(\mathcal{R}, \xi)$, we set
    \[
    \sigma_{\mathcal{R}^*, \xi} = \mathrm{ran}(\pi^{\mathcal{R}, \xi}_{\mathcal{R}^*, \infty})\cap\alpha\in\power_{\omega_1}(\alpha)
    \]
    and define
    \[
    C_{\mathcal{R}, \xi}  = \{\sigma_{\mathcal{R}^*, \xi}\mid\mathcal{R}^*\in\mathcal{F}^*_k(\mathcal{R}, \xi)\land\alpha\in\mathrm{ran}(\pi^{\mathcal{R}, \xi}_{\mathcal{R}^*, \infty})\}.
    \]
\begin{claim}\label{finding_clubs}
    Whenever $\mathcal{R}$ is a genericity iterate of $\mathcal{Q}'$ above $\eta$, $k\subset\mathrm{Col}(\omega, {<}\delta)$ is a maximal $\mathcal{R}$-generic, and $\xi\in[\eta, \delta)$, then $C_{\mathcal{R}, \xi}$ contains a club subset of $\power_{\omega_1}(\alpha)$ in $\mathcal{V}_{\mathcal{R}}[k]$.
\end{claim}
\begin{proof}
    Since $\alpha<\delta^{\mathcal{Q}', \eta}_{\infty}=\delta^{\mathcal{R}, \eta}_{\infty}<(\delta^+)^{\mathcal{V}_{\mathcal{R}}[k]}$, there is a bijection $f\colon\delta\to\power_{\omega_1}(\alpha)$ in $\mathcal{V}_{\mathcal{R}}[k]$.
    Using such an $f$, we inductively define $\mathcal{R}_{\beta}\in\mathcal{F}^*_k(\mathcal{R}, \xi)$ for $\beta<\delta$ as follows.
    First note that $\delta^{\mathcal{Q}', \eta}_{\infty}=\delta^{\mathcal{R}, \xi}_{\infty}$ as $(\mathcal{Q}, \eta)$ stabilizes $\delta_{\infty}$.
    It follows that there are cofinally many $\mathcal{R}^*$ in $\mathcal{F}^*_k(\mathcal{R}, \xi)$ such that $\alpha\in\mathrm{ran}(\pi^{\mathcal{R}, \xi}_{\mathcal{R}^*, \infty})$.
    Now let $\mathcal{R}_0\in\mathcal{F}^*_k(\mathcal{R}, \xi)$ be such an $\mathcal{R}^*$.
    Also, for each $\beta<\delta$, let $\mathcal{R}_{\beta+1}\in\mathcal{F}^*_k(\mathcal{R}, \xi)$ be an iterate of $\mathcal{R}_{\beta}$ such that $f(\beta)\subset\mathrm{ran}(\pi^{\mathcal{R}, \xi}_{\mathcal{R}_{\beta+1}, \infty})$.
    This is possible because for any $\sigma\in\power_{\omega_1}(\alpha)$, there are cofinally many $\mathcal{R}^*$ in $\mathcal{F}^*_k(\mathcal{R}, \xi)$ such that $\sigma\subset\mathrm{ran}(\pi^{\mathcal{R}, \xi}_{\mathcal{R}^*, \infty})$.
    Finally, for each limit ordinal $\lambda<\delta$, let $\mathcal{R}_{\lambda}$ be the direct limit of $\langle\mathcal{R}_\beta, \pi_{\mathcal{R}_{\beta}, \mathcal{R}_{\gamma}}\mid\beta<\gamma<\lambda\rangle$.
    By the construction, $\alpha\in\mathrm{ran}(\pi^{\mathcal{R}, \xi}_{\mathcal{R}_{\beta}, \infty})$ for any $\beta<\delta$ and $\{\sigma_{\mathcal{R}_{\beta}, \xi}\mid\beta<\delta\}$ is a closed unbounded subset of $\power_{\omega_1}(\alpha)$.
\end{proof}

\begin{claim}\label{club_dichotomy}
    Let $A\subset\power_{\omega_1}(\alpha)$ be such that $A\in\mathsf{CDM}^+(\mathcal{Q}', \eta)\vert\gamma$.
    Then there are a genericity iterate $\mathcal{R}$ of $\mathcal{Q}'$ and a $\xi\in [\eta, \delta)$ such that $\mathcal{R}\in\mathcal{V}_{\mathcal{Q}'}[h']$ and the following hold:
    \begin{enumerate}
        \item If $\sigma_{\mathcal{R}, \xi}\in A$, then $C_{\mathcal{R}, \xi}\subset A$, and
        \item If $\sigma_{\mathcal{R}, \xi}\notin A$, then $C_{\mathcal{R}, \xi}\subset\power_{\omega_1}(\alpha)\setminus A$.
    \end{enumerate}
\end{claim}
\begin{proof}
    Let $A\subset\power_{\omega_1}(\alpha)$ be in $\mathsf{CDM}^+(\mathcal{Q}', \eta)$.
    Then for some formula $\phi$ in the language for $\mathsf{CDM}^+(\mathcal{Q}', \eta)$ and some ordinal $\overline{\gamma}<\gamma$,
    \[
    A=\{\sigma\in\power_{\omega_1}(\alpha)\mid(\mathsf{CDM}^+(\mathcal{Q}', \eta)\vert\overline{\gamma}; {\in}, \vec{\mu})\models\phi[\sigma, Y, Z, x, \vec{\beta}]\},
    \]
    where $Y\in {}^{\omega}\zeta$ for some $\zeta<\delta^{\mathcal{Q}', \eta}_{\infty}$, $Z\in\Gamma^*_g$, $x\in\mathbb{R}^*_g$, and $\vec{\beta}\in{}^{<\omega}\overline{\gamma}$.
    Then let $\mathcal{R}$ be a genericity iterate of $\mathcal{Q}'$ above $\eta$ such that $\mathcal{R}\in\mathcal{V}_{\mathcal{Q}'}[h']$ and $\{\alpha, \vec{\beta}, \overline{\gamma}\}\cup\mathrm{ran}(Y)\subset\mathrm{ran}(\pi^{\mathcal{R}, \eta}_{\mathcal{V}_{\mathcal{R}}, \infty})$.
    To find such an $\mathcal{R}$, let $\mathcal{Q}^*\in\mathcal{F}^*_{h'}(\mathcal{Q}', \eta)$ be such that $\{\alpha, \vec{\beta}, \overline{\gamma}\}\cup\mathrm{ran}(Y)\subset\mathrm{ran}(\pi^{\mathcal{Q}', \eta}_{\mathcal{V}_{\mathcal{Q}^*}, \infty})$.
    Such a $\mathcal{Q}^*$ exists because $\mathcal{F}^*_{h'}(\mathcal{Q}', \eta)$ is countably directed.
    By \cref{easy_lemma}(1), there is an iterate $\mathcal{R}$ of $\mathcal{Q}^*$ in $\mathcal{V}_{\mathcal{Q}'}[h']$ such that it is a genericity iterate of $\mathcal{Q}'$ and $\mathcal{T}_{\mathcal{Q}', \mathcal{Q}^*}{}^{\frown}\mathcal{T}_{\mathcal{Q}^*, \mathcal{R}}$ is normal.
    Since $\pi^{\mathcal{Q}', \eta}_{\mathcal{V}_{\mathcal{Q}^*}, \infty}=\pi^{\mathcal{R}, \eta}_{\mathcal{V}_{\mathcal{R}}, \infty}\circ\pi_{\mathcal{V}_{\mathcal{Q}^*}, \mathcal{V}_{\mathcal{R}}}$, $\mathcal{R}$ satisfies the desired property.
    This argument to find a genericity iterate that ``catches'' a given countable subset of the direct limit model is repeatedly used in this paper.
    For what follows, we will use the next subclaim.

    \begin{subclaim}[{\cite[Lemma 4.3]{CoveringChang}}]\label{stabilizing_parameters}
        Whenever $\mathcal{S}$ is a genericity iterate of $\mathcal{R}$ above $\eta$, if $a\in\mathrm{ran}(\pi^{\mathcal{R}, \eta}_{\mathcal{V}_{\mathcal{R}}, \infty})$ then $\pi_{\mathcal{V}_{\mathcal{R}}, \mathcal{V}_{\mathcal{S}}}(a)=a$.
        In particular, $\pi_{\mathcal{V}_{\mathcal{R}}, \mathcal{V}_{\mathcal{S}}}((\alpha, \vec{\beta}, \overline{\gamma}))=(\alpha, \vec{\beta}, \overline{\gamma})$ and $\pi_{\mathcal{V}_{\mathcal{R}}, \mathcal{V}_{\mathcal{S}}}(Y(i))=Y(i)$ for any $i<\omega$.
    \end{subclaim}
    \begin{proof}
        Let $a_{\mathcal{R}}=(\pi^{\mathcal{R}, \eta}_{\mathcal{V}_{\mathcal{R}}, \infty})^{-1}(a)$.
        Then we have
        \[
        \pi_{\mathcal{V}_{\mathcal{R}}, \mathcal{V}_{\mathcal{S}}}(a)
        =\pi_{\mathcal{V}_{\mathcal{R}}, \mathcal{V}_{\mathcal{S}}}(\pi^{\mathcal{R}, \eta}_{\mathcal{V}_{\mathcal{R}}, \infty}(a_{\mathcal{R}}))
        =\pi^{\mathcal{S}, \eta}_{\mathcal{V}_{\mathcal{S}}, \infty}(\pi_{\mathcal{V}_{\mathcal{R}}, \mathcal{V}_{\mathcal{S}}}(a_{\mathcal{R}}))
        =\pi^{\mathcal{R}, \eta}_{\mathcal{V}_{\mathcal{R}}, \infty}(a_{\mathcal{R}})=a.
        \]
        The second equality follows from the elementarity of $\pi_{\mathcal{V}_{\mathcal{R}}, \mathcal{V}_{\mathcal{S}}}$ and the third equality holds since $\pi^{\mathcal{R}, \eta}_{\mathcal{V}_{\mathcal{R}}, \infty}=\pi^{\mathcal{S}, \eta}_{\mathcal{V}_{\mathcal{S}}, \infty}\circ\pi_{\mathcal{V}_{\mathcal{R}}, \mathcal{V}_{\mathcal{S}}}$ by \cref{pres_of_M_infty}.
    \end{proof}

    Let $k\subset\mathrm{Col}(\omega, {<}\delta)$ be a maximal $\mathcal{R}$-generic such that $k\in\mathcal{Q}'[h']$.
    Since $Y=\langle Y(i)\mid i<\omega\rangle\in {}^{\omega}\zeta$ for some $\zeta<\delta^{\mathcal{Q}', \eta}_{\infty}$, we can fix a $\xi_Y<\delta$ such that $\mathrm{ran}(Y)\subset\pi^{\mathcal{R}, \eta}_{\mathcal{R}, \infty}[\xi_Y]$.
    Let $y\in\mathbb{R}^*_{k}$ code a function $f_y\colon\omega\to\xi_Y$ such that for any $i\in\omega$,
    \[
    Y(i)=\pi^{\mathcal{R}, \eta}_{\mathcal{R}, \infty}(f_y(i)).
    \]
    Also, since $\{\mathrm{Code}(\Sigma^g_{\mathcal{P}\vert\xi})\mid\xi<\delta\}$\footnote{For an iteration strategy $\Sigma$ for a countable structure, $\mathrm{Code}(\Sigma)$ is a set of reals that canonically codes $\Sigma\uphar\mathrm{HC}$, where $\mathrm{HC}$ denotes the set of hereditarily countable sets. See \cite[Section 2.7]{CPMP}.} is Wadge cofinal in $\Gamma^*_g$ as argued in the proof of \cite[Proposition 4.2]{CoveringChang}, we may assume that $Z=\mathrm{Code}(\Sigma^g_{\mathcal{P}\vert\xi_Z})$ for some $\xi_Z<\delta$.
    Let $z\in\mathbb{R}^*_k$ be a real coding $\pi_{\mathcal{P}, \mathcal{R}}\uphar(\mathcal{P}\vert\xi_Z)\colon\mathcal{P}\vert\xi_Z\to\mathcal{R}\vert\pi_{\mathcal{P}, \mathcal{R}}(\xi_Z)$.
    Note that $Z$ can be defined from $z$ as the code of the $\pi_{\mathcal{P}, \mathcal{R}}$-pullback of the strategy for $\mathcal{R}\vert\pi_{\mathcal{P}, \mathcal{R}}(\xi_Z)$ determined by the strategy predicate of $\mathcal{R}$.
    Then we can fix some $\xi\in[\max\{\eta, \xi_Y, \pi_{\mathcal{P}, \mathcal{R}}(\xi_Z)\}, \delta)$ such that $x, y, z\in\mathcal{R}[k\uphar\xi]$.

    We now begin the main argument in the proof of \cref{club_dichotomy}.
    Variants of this argument will be used repeatedly throughout the paper, and when we refer to ``as in the proof of \cref{club_dichotomy},'' we mean the argument that follows.
    To show (1) in the statement of \cref{club_dichotomy}, suppose that $\sigma_{\mathcal{R}, \xi}\in A$.
    Then
    \[
    \mathcal{V}_{\mathcal{R}}[x, y, z]\models\phi^*[\mathrm{ran}(\pi^{\mathcal{R}, \xi}_{\mathcal{R}^*, \infty})\cap\alpha, x, y, z, \eta, \delta, \vec{\beta}, \overline{\gamma}],
    \]
    where the formula $\phi^*$ is the conjunction of the following.\footnote{In general, for any transitive model $M$ of $\mathsf{ZF}$ and any subset $a$ of an element of $M$, $M[a]$ denotes a transitive minimal model $N$ of $\mathsf{ZF}$ such that $M\cup\{a\}\subset N$ and $M\cap\ord=N\cap\ord$, if such an $N$ exists.
    So $\mathcal{V}_{\mathcal{R}}[x, y, z]$ makes sense because $x, y, z$ are in some generic extension of $\mathcal{R}$.}
    \begin{itemize}
        \item $y$ codes a function $f\colon\omega\to\zeta$ for some $\zeta<\delta$, and
        \item $z$ codes an elementary embedding $\pi\colon\mathcal{M}\to\mathcal{N}$ for some lbr hod premice $\mathcal{M}$ and $\mathcal{N}$ with $\mathcal{N}\trianglelefteq\mathcal{R}$, and 
        \item letting $Y=\langle\pi^{\mathcal{R}, \eta}_{\mathcal{R}, \infty}(f(i))\mid i\in\omega\rangle$ and $Z$ be the code of the $\pi$-pullback of the strategy for $\mathcal{N}$ determined by the strategy predicate of $\mathcal{R}$, the empty condition of $\mathrm{Col}(\omega, {<}\delta)$ forces that
        \[
        (\mathsf{CDM}^+(\mathcal{R}, \eta)\vert\overline{\gamma}; {\in}, \vec{\mu})\models\phi[\mathrm{ran}(\pi^{\mathcal{R}, \xi}_{\mathcal{R}^*, \infty})\cap\alpha, Y, Z, x, \vec{\beta}].
        \]
    \end{itemize}

    Now let $\mathcal{R}^*\in\mathcal{F}^*_k(\mathcal{R}, \xi)$ be such that $\alpha\in\mathrm{ran}(\pi^{\mathcal{R}, \xi}_{\mathcal{R}^*, \infty})$.
    Since $\alpha<\delta^{\mathcal{Q}', \eta}_{\infty}=\delta^{\mathcal{R}, \xi}_{\infty}$ as $(\mathcal{Q}, \eta)$ stabilizes $\delta_{\infty}$, we have $\alpha_{\mathcal{R}^*}:=(\pi^{\mathcal{R}, \xi}_{\mathcal{R}^*, \infty})^{-1}(\alpha)<\delta$.
    By \cref{easy_lemma}(1), there is an iterate $\mathcal{S}$ of $\mathcal{R}^*$ in $\mathcal{R}[k]$ such that $\mathcal{S}$ is a genericity iterate of $\mathcal{R}$, $\mathcal{T}_{\mathcal{R}, \mathcal{R}^*}{}^{\frown}\mathcal{T}_{\mathcal{R}^*, \mathcal{S}}$ is normal and $\crit(\pi_{\mathcal{R}^*, \mathcal{S}})>\alpha_{\mathcal{R}^*}$.
    By \cref{stabilizing_parameters}, the elementarity of $\pi^+_{\mathcal{V}_{\mathcal{R}}, \mathcal{V}_{\mathcal{S}}}\colon\mathcal{V}_{\mathcal{R}}[x, y, z]\to\mathcal{V}_{\mathcal{S}}[x, y, z]$, which is the canonical liftup of $\pi_{\mathcal{V}_{\mathcal{R}}, \mathcal{V}_{\mathcal{S}}}$, implies that
    \[
    \mathcal{V}_{\mathcal{S}}[x, y, z]\models\phi^*[\mathrm{ran}(\pi^{\mathcal{S}, \xi}_{\mathcal{S}, \infty})\cap\alpha, x, y, z, \eta, \delta, \vec{\beta}, \overline{\gamma}].
    \]
    Then the following observations imply $\sigma_{\mathcal{S}, \xi}\in A$:
    \begin{itemize}
        \item Since $\mathcal{S}$ is a genericity iterate of $\mathcal{R}$ above $\eta$,
        \[
        (\mathsf{CDM}^{+}(\mathcal{R}, \eta)\vert\overline{\gamma})^{\mathcal{V}_{\mathcal{R}}[k]}=(\mathsf{CDM}^{+}(\mathcal{S}, \eta)\vert\overline{\gamma})^{\mathcal{V}_{\mathcal{S}}[l]},
        \]
        where $l\subset\mathrm{Col}(\omega, {<}\delta)$ is a maximal $\mathcal{S}$-generic, by \cref{pres_of_CDM_plus} and the induction hypothesis.
        \item Let $Y'=\langle\pi^{\mathcal{S}, \eta}_{\mathcal{S}, \infty}(f_y(i))\mid i\in\omega\rangle$, where $f_y\colon\omega\to\xi_Y$ is the function coded by $y$.
        Then
        \[
        Y'=\pi_{\mathcal{V}_{\mathcal{R}}, \mathcal{V}_{\mathcal{S}}}(Y)=Y.
        \]
        The first equality here follows from \cref{pres_of_M_infty} and $\crit(\pi_{\mathcal{R}, \mathcal{S}})>\xi_Y$.
        The second equality holds by \cref{stabilizing_parameters}.
        \item Since $\crit(\pi_{\mathcal{R}, \mathcal{S}})>\pi_{\mathcal{P}, \mathcal{R}}(\xi_Z)$, $\mathcal{R}\vert\pi_{\mathcal{P}, \mathcal{R}}(\xi_Z)=\mathcal{S}\vert\pi_{\mathcal{P}, \mathcal{R}}(\xi_Z)$, so $z$ codes an elementary embedding into an initial segment of $\mathcal{S}$.
        Also, as $\Sigma_{\mathcal{R}\vert\pi_{\mathcal{P}, \mathcal{R}}(\xi_Z)}=\Sigma_{\mathcal{S}\vert\pi_{\mathcal{P}, \mathcal{R}}(\xi_Z)}$, the same $Z$ is obtained from $z$ over both $\mathcal{V}_{\mathcal{R}}$ and $\mathcal{V}_{\mathcal{S}}$.
    \end{itemize}
    As $\crit(\pi_{\mathcal{R}^*, \mathcal{S}})>\alpha_{\mathcal{R}^*}$ and $\pi^{\mathcal{R}, \xi}_{\mathcal{R}^*, \infty}=\pi^{\mathcal{S}, \xi}_{\mathcal{S}, \infty}\circ\pi_{\mathcal{R}^*, \mathcal{S}}$, we have
    \[
    \sigma_{\mathcal{R}^*, \xi}=\pi^{\mathcal{R}, \xi}_{\mathcal{R}^*, \infty}[\alpha_{\mathcal{R}^*}]=\pi^{\mathcal{S}, \xi}_{\mathcal{S}, \infty}[\alpha_{\mathcal{R}^*}]=\sigma_{\mathcal{S}, \xi}.
    \]
    So we get $\sigma_{\mathcal{R}^*, \xi}\in A$.
    Therefore, $C_{\mathcal{R}, \xi}\subset A$.
    The same argument when $\phi^*$ with $\neg\phi^*$ shows (2).
\end{proof}

\begin{claim}\label{agreement_between_club_filters}
    $\mu_{\alpha}\cap\mathsf{CDM}^+(\mathcal{Q}', \eta)\vert\gamma=\mu_{\alpha}^{\mathcal{V}_{\mathcal{Q'}}[h']}\cap\mathsf{CDM}^+(\mathcal{Q}', \eta)\vert\gamma$.
\end{claim}
\begin{proof}
    Let $A\subset\power_{\omega_1}(\alpha)$ be in $\mathsf{CDM}^+(\mathcal{Q}', \eta)\vert\gamma$.
    Take a genericity iterate $\mathcal{R}$ of $\mathcal{Q}'$, a maximal $\mathcal{R}$-generic $k\subset\mathrm{Col}(\omega, {<}\delta)$, and $\xi\in[\eta, \delta)$ such that the conclusion of \cref{club_dichotomy} holds.
    Note that $C_{\mathcal{R}, \xi}$ contains a club subset of $\power_{\omega_1}(\alpha)$ in $\mathcal{V}_{\mathcal{R}}[k]\subset\mathcal{V}_{\mathcal{Q}'}[h']$ by \cref{finding_clubs}.
    Therefore, if $A\in\mu_{\alpha}^{\mathcal{V}_{\mathcal{Q}'}[h']}$, then (1) of \cref{club_dichotomy} holds and thus $A\in\mu_{\alpha}$.
    On the other hand, if $A\notin\mu_{\alpha}^{\mathcal{V}_{\mathcal{Q}'}[h']}$, then (2) of \cref{club_dichotomy} holds and thus $\power_{\omega_1}(\alpha)\setminus A\in\mu_{\alpha}$.
    Since $\mu_{\alpha}$ is a filter, $A\notin\mu_{\alpha}$.
\end{proof}

This completes the proof of \cref{locality_of_CDM_plus}.
\end{proof}

\begin{thm}\label{spct_meas_in_CDM_plus}
    Let $\mathcal{Q}$ be a genericity iteration of $\mathcal{P}$ and let $\eta<\delta$ be such that $(\mathcal{Q}, \eta)$ stabilizes $\delta_{\infty}$.
    Then for each $\alpha\in[\delta, \delta^{\mathcal{Q}, \eta}_{\infty})$, $\mu_{\alpha}\cap\mathsf{CDM}^{+}(\mathcal{Q}, \eta)$ is a supercompact measure on $\power_{\omega_1}(\alpha)$ in $\mathsf{CDM}^{+}(\mathcal{Q}, \eta)$.
\end{thm}
\begin{proof}
    It is obvious that $\mu_{\alpha}\cap\mathsf{CDM}^{+}(\mathcal{Q}, \eta)$ is a filter.
    The proof of \cref{agreement_between_club_filters} shows that $\mu_{\alpha}\cap\mathsf{CDM}^{+}(\mathcal{Q}, \eta)$ is an ultrafilter on $\power_{\omega_1}(\alpha)$. Fineness of $\mu_{\alpha}\cap\mathsf{CDM}^{+}(\mathcal{Q}, \eta)$ follows from fineness of $\mu_{\alpha}$ because for any $\xi\in\alpha$, the set $\{\sigma\in\power_{\omega_1}(\alpha)\mid\xi\in\sigma\}$ is in $\mathsf{CDM}^+(\mathcal{Q}, \eta)$. 
    Also, it is easy to see countable completeness of $\mu_{\alpha}\cap\mathsf{CDM}^{+}(\mathcal{Q}, \eta)$: Whenever $\langle A_n\mid n<\omega\rangle\in\mathsf{CDM}^+(\mathcal{Q}, \eta)$ is such that $A_n\in\mu_{\alpha}\cap\mathsf{CDM}^{+}(\mathcal{Q}, \eta)$ for all $n<\omega$, then $\bigcap_{n<\omega}A_n\in\mathsf{CDM}^+(\mathcal{Q}, \eta)$ and it is also in $\mu_{\alpha}$ by countable completeness of $\mu_{\alpha}$.
    Similarly, normality of $\mu_{\alpha}\cap\mathsf{CDM}^{+}(\mathcal{Q}, \eta)$ follows from normality of $\mu_{\alpha}$.
    Therefore, $\mu_{\alpha}\cap\mathsf{CDM}^{+}(\mathcal{Q}, \eta)$ is a supercompact measure in $\mathsf{CDM}^+(\mathcal{Q}, \eta)$.
\end{proof}

\begin{rem}
    Woodin showed that $\mathsf{AD}^+$ implies that the club filter on $\power_{\omega_1}(\alpha)$ is a supercompact measure for any $\alpha<\Theta$.\footnote{To the best of our knowledge, the full proof of this theorem is not written anywhere.}
    So, \cref{spct_meas_in_CDM_plus} is not new if $\delta^{\mathcal{Q}, \eta}_{\infty}=\Theta^{\mathsf{CDM}^{+}(\mathcal{Q}, \eta)}$.
    However, we will show as \cref{main_thm_2} below that $\delta^{\mathcal{Q}, \eta}_{\infty}>\Theta^{\mathsf{CDM}^{+}(\mathcal{Q}, \eta)}$ assuming that $\delta$ is a limit of Woodin cardinals that is also a limit of ${<}\delta$-strong cardinals.
\end{rem}

\begin{rem}\label{rem_on_def_of_CDM_plus}
If we would have defined $\mathsf{CDM}^+(\mathcal{Q}, \eta)$ as $\mathsf{CDM}$ with the club measures on $\power_{\omega_1}({}^{\omega}\alpha)$ for $\alpha<\delta^{\mathcal{Q}, \eta}_{\infty}$, then \cref{finding_clubs} would fail: Let $\sigma_{\mathcal{R}^*, \xi} = \mathrm{ran}(\pi^{\mathcal{R}, \xi}_{\mathcal{R}^*, \infty})\cap{}^{\omega}\alpha$.
Then the set in \cref{finding_clubs} cannot be unbounded because if $f\in{}^{\omega}\alpha\setminus\mathcal{M}_{\infty}$ then $f\notin\sigma_{\mathcal{R}^*, \xi}$ for any $\mathcal{R}^*$.
Also, if one changes the definition of $\sigma_{\mathcal{R}^*, \xi}$ to take this issue into account, then closedness would be a new problem.
\end{rem}

\begin{ques}
    Is there a variant of $\mathsf{CDM}$ where $\omega_1$ is ${}^{\omega}\alpha$-supercompact for all $\alpha<\delta^{\mathcal{Q}, \eta}_{\infty}$?
\end{ques}

\subsection{Proof of Determinacy}

First, note that by \cref{pres_of_CDM_plus,locality_of_CDM_plus}, we have the following.

\begin{lem}\label{pres_of_local_CDM_plus}
Let $\mathcal{Q}$ be a genericity iterate of $\mathcal{P}$ and let $\eta<\delta$ be such that $(\mathcal{Q}, \eta)$ stabilizes $\delta_{\infty}$.
Also, let $h\subset\mathrm{Col}(\omega, {<}\delta)$ be a maximal $\mathcal{Q}$-generic.
Then whenever $\mathcal{R}$ is a genericity iterate of $\mathcal{Q}$ above $\eta$ and $k\subset\mathrm{Col}(\omega, {<}\delta)$ is a maximal $\mathcal{R}$-generic,
\[
\mathsf{CDM}^{+}(\mathcal{Q}, \eta)=\mathsf{CDM}^{+}(\mathcal{Q}, \eta)^{\mathcal{V}_{\mathcal{Q}}[h]}=\mathsf{CDM}^+(\mathcal{R}, \eta)^{\mathcal{V}_{\mathcal{R}}[k]}.
\]
\end{lem}

Thanks to \cref{pres_of_local_CDM_plus}, we can get the following theorem by the proof of the main theorem of \cite{CoveringChang}.

\begin{thm}\label{sets_of_reals_in_CDM_plus}
Let $\mathcal{Q}$ be a genericity iterate of $\mathcal{P}$ and 
let $\eta<\delta$ be such that $(\mathcal{Q}, \eta)$ stabilizes $\delta_{\infty}$.
Then
\[
\mathsf{CDM}^{+}(\mathcal{Q}, \eta)\cap\power(\mathbb{R}^*_g)=\Gamma^*_g.
\]
\end{thm}
\begin{proof}
We work in $\mathcal{V}(\mathbb{R}^*_g)$.
Let $A\subset\mathbb{R}^*_g$ be in $\mathsf{CDM}^{+}(\mathcal{Q}, \eta)$.
Then for some formula $\phi$ in the language for $\mathsf{CDM}^+(\mathcal{Q}, \eta)$ and for some ordinal $\gamma$,
\[
A=\{u\in\mathbb{R}^*_g\mid(\mathsf{CDM}^{+}(\mathcal{Q}, \eta)\vert\gamma; {\in}, \vec{\mu})\models\phi[u, Y, Z, x, \vec{\beta}]\},
\]
where $Y=\langle Y(i)\mid i<\omega\rangle\in {}^{\omega}\zeta$ for some $\zeta<\delta^{\mathcal{Q}, \eta}_{\infty}$, $Z\in\Gamma^*_g$, $x\in\mathbb{R}^*_g$, and $\vec{\beta}\in{}^{<\omega}\gamma$.
Then we can take a genericity iterate $\mathcal{R}$ of $\mathcal{Q}$ above $\eta$ such that $\{\vec{\beta}, \gamma\}\cup\mathrm{ran}(Y)\subset\mathrm{ran}(\pi^{\mathcal{R}, \eta}_{\mathcal{V}_{\mathcal{R}}, \infty})$.
The proof of \cref{stabilizing_parameters} shows the following claim.

\begin{claim}\label{stabilizing_parameters_1}
Whenever $\mathcal{S}$ is a genericity iterate of $\mathcal{R}$ above $\eta$, $\pi_{\mathcal{V}_{\mathcal{R}}, \mathcal{V}_{\mathcal{S}}}((\vec{\beta}, \gamma))=(\vec{\beta}, \gamma)$ and $\pi_{\mathcal{V}_{\mathcal{R}}, \mathcal{V}_{\mathcal{S}}}(Y(i))=Y(i)$ for any $i<\omega$.
\end{claim}

Let $k\subset\mathrm{Col}(\omega, {<}\delta)$ be a maximal $\mathcal{R}$-generic.
Since $Y=\langle Y(i)\mid i<\omega\rangle\in {}^{\omega}\xi$ for some $\xi<\delta^{\mathcal{Q}, \eta}_{\infty}$, we can take $\xi_Y<\delta$ such that $\mathrm{ran}(Y)\subset\pi^{\mathcal{R}, \eta}_{\mathcal{R}, \infty}[\xi_Y]$.
Let $y\in\mathbb{R}^*_{k}$ code a function $f_y\colon\omega\to\xi_Y$ such that for any $i\in\omega$,
\[
Y(i)=\pi^{\mathcal{R}, \eta}_{\mathcal{R}, \infty}(f_y(i)).
\]
Also, we may assume that $Z=\mathrm{Code}(\Sigma^g_{\mathcal{P}\vert\xi_Z})$ for some $\xi_Z<\delta$.
Let $z\in\mathbb{R}^*_k$ be a real coding $\pi_{\mathcal{P}, \mathcal{R}}\uphar(\mathcal{P}\vert\xi_Z)\colon\mathcal{P}\vert\xi_Z\to\mathcal{R}\vert\pi_{\mathcal{P}, \mathcal{R}}(\xi_Z)$.
Then fix any $\eta'\in[\max\{\eta, \xi_Y, \pi_{\mathcal{P}, \mathcal{R}}(\xi_Z)\}, \delta)$ such that $x, y, z\in\mathcal{R}[k\uphar\eta']$.

Because $\mathsf{CDM}^+(\mathcal{Q}, \eta)=\mathsf{CDM}^+(\mathcal{R}, \eta)^{\mathcal{V}_{\mathcal{R}}[k]}$ by \cref{pres_of_local_CDM_plus}, we have
\[
A=\{u\in\mathbb{R}^*_g\mid\mathcal{V}_{\mathcal{R}}[x, y, z][u]\models\phi^*[u, x, y, z, \eta, \delta, \vec{\beta}, \gamma]\},
\]
where the formula $\phi^*$ is obtained as in the proof of \cref{club_dichotomy}.

Let $\delta'<\delta$ be the least Woodin cardinal of $\mathcal{R}$ above $\eta'$.
For any non-dropping iterate $\mathcal{R}^*$ of $\mathcal{R}$ above $\eta$, there is a unique standard\footnote{A $\mathbb{P}$-term $\tau$ over $M$ for a set of reals is called \emph{standard} if
\[
\tau=\{\langle p, \sigma\rangle\mid \sigma\subset\mathbb{P}\times\{\check{n}\mid n\in\omega\}\text{ and }p\Vdash^{M}_{\mathbb{P}}\sigma\in\tau\}.
\]
For any standard $\mathbb{P}$-terms $\tau_0$ and $\tau_1$ over $M$, if $\tau_0^{g}=\tau_1^{g}$ for any $M$-generic $g\subset\mathbb{P}$, then $\tau_0=\tau_1$.} $\mathrm{Col}(\omega, \pi_{\mathcal{R}, \mathcal{R}^*}(\delta'))$-term $\tau_{\mathcal{R}^*}$ over $\mathcal{R}^*[x, y, z]$ such that whenever $k^*\subset\mathrm{Col}(\omega, \pi_{\mathcal{R}, \mathcal{R}^*}(\delta'))$ is $\mathcal{R}^*$-generic,
\[
(\tau_{\mathcal{R}^*})^{k^*}=\{u\in\mathbb{R}^{\mathcal{R}^*[x, y, z][k^*]}\mid\mathcal{V}_{\mathcal{R}^*}[x, y, z][k^*]\models\phi^*[u, x, y, z, \eta, \delta, \vec{\beta}, \gamma]\}.
\]
Now we argue that this characterizes the set $A\subset\mathbb{R}^*_g$ we started with.

\begin{claim}[{\cite[Lemma 4.4]{CoveringChang}}]\label{projective_def_of_A}
For any $u\in\mathbb{R}^*_g$, $u\in A$ if and only if whenever $\mathcal{R}^*\in\mathcal{F}^*_g(\mathcal{R}, \eta')$ as witnessed by an iteration tree based on $\mathcal{R}\vert\delta'$ and $k^*\subset\mathrm{Col}(\omega, \pi_{\mathcal{R}, \mathcal{R}^*}(\delta'))$ is $\mathcal{R}^*[x, y, z]$-generic with $u\in\mathcal{R}^*[x, y, z][k^*]$, then $u\in (\tau_{\mathcal{R}^*})^{k^*}$.
\end{claim}
\begin{proof}
To show the forward direction, suppose that $u\in A$ and let $\mathcal{R}^*$ and $k^*$ be as in the claim.
Then take an iterate $\mathcal{S}$ of $\mathcal{R}^*$ above $\pi_{\mathcal{R}, \mathcal{R}^*}(\delta')$ that is a genericity iterate of $\mathcal{R}$.
Then as in the proof of \cref{club_dichotomy}, it follows from \cref{pres_of_local_CDM_plus} and \cref{stabilizing_parameters_1} that $u\in (\tau_{\mathcal{S}})^{k^*}$.
Let $\pi_{\mathcal{R}^*, \mathcal{S}}^{+}\colon\mathcal{R}^*[x, y, z]\to\mathcal{S}[x, y, z]$ be the canonical liftup of $\pi_{\mathcal{R}^*, \mathcal{S}}$.
Since $\crit(\pi_{\mathcal{R}^*, \mathcal{S}}^{+})$ $>\pi_{\mathcal{R}, \mathcal{R}^*}(\delta')$, $\tau_{\mathcal{R}^*}=\pi_{\mathcal{R}^*, \mathcal{S}}^{+}(\tau_{\mathcal{R}^*})$.
On the other hand, $\pi_{\mathcal{R}^*, \mathcal{S}}^{+}(\tau_{\mathcal{R}^*})=\tau_{\mathcal{S}}$, because $\pi_{\mathcal{R}^*, \mathcal{S}}^+$ does not move any parameters in the definition of $\tau_{\mathcal{R}^*}$ by \cref{stabilizing_parameters_1}.
It follows that $\tau_{\mathcal{R}^*} = \tau_{\mathcal{S}}$, which implies $u\in(\tau_{\mathcal{R}^*})^{k^*}$.

The proof of the reverse direction is very similar.
Let $\mathcal{R}^*$ and $k^*$ be as in the claim and suppose that $u\in(\tau_{\mathcal{R}^*})^{k^*}$.
Take a genericity iterate $\mathcal{S}$ of $\mathcal{R}$ as before.
Then we have $\tau_{\mathcal{R}^*}=\tau_{\mathcal{S}}$ by the same argument as before and thus $u\in(\tau_{\mathcal{S}})^{k^*}$.
Unravelling the definition of $\tau_{\mathcal{S}}$, we get $u\in A$.
\end{proof}

By \cref{projective_def_of_A}, $A$ is projective in $\mathrm{Code}(\Sigma^g_{\mathcal{R}\vert\delta'})$.
Since $\Sigma^g_{\mathcal{R}\vert\delta'}$ is a tail strategy of $\Sigma^g_{\mathcal{P}\vert\xi}$ for some $\xi<\delta$, $\mathrm{Code}(\Sigma^g_{\mathcal{R}\vert\delta'})$ is projective in $\mathrm{Code}(\Sigma^g_{\mathcal{P}\vert\xi})\in\Gamma^*_g$.
Therefore, $\mathrm{Code}(\Sigma^g_{\mathcal{R}\vert\delta'})\in\Gamma^*_g$, which in turn implies that $A\in\Gamma^*_g$.
This completes the proof of \cref{sets_of_reals_in_CDM_plus}.
\end{proof}

The following corollary is an immediate consequence of \cref{determinacy_in_DM} and \cref{sets_of_reals_in_CDM_plus}.

\begin{cor}\label{determinacy_in_CDM_plus_1}
Let $\mathcal{Q}$ be a genericity iterate of $\mathcal{P}$ and 
let $\eta<\delta$ be such that $(\mathcal{Q}, \eta)$ stabilizes $\delta_{\infty}$.
Then
\[\mathsf{CDM}^{+}(\mathcal{Q}, \eta)\models\mathsf{AD}^{+}+\mathsf{AD}_{\mathbb{R}}.
\]
\end{cor}

\subsection{The regularity of $\Theta$}

Using the proof of the main theorem of \cite{dm_of_self_it}, we show the following result.

\begin{thm}\label{determinacy_in_CDM_plus_2}
Suppose that $\delta$ is a regular limit of Woodin cardinals in $\mathcal{V}$.
Let $\mathcal{Q}$ be a genericity iterate of $\mathcal{P}$ and 
let $\eta<\delta$ be such that $(\mathcal{Q}, \eta)$ stabilizes $\delta_{\infty}$.
Then
\[\mathsf{CDM}^{+}(\mathcal{Q}, \eta)\models\mathsf{DC}+\Theta\text{ is regular.}
\]
\end{thm}
\begin{proof}
First, we show that $\mathsf{CDM}^+(\mathcal{Q}, \eta)\models\cf(\Theta)>\omega$.
This follows from the proof of \cite[Corollary 3.7]{dm_of_self_it} without any change as follows.
Suppose toward a contradiction that $\mathsf{CDM}^+(\mathcal{Q}, \eta)\models\cf(\Theta)=\omega$.
We work in $\mathcal{V}[g]$ for now.
Then by \cref{sets_of_reals_in_CDM_plus}, there is a sequence $\langle A_n \mid n<\omega\rangle$ that is Wadge cofinal in $\Gamma^*_g$.
For any $n<\omega$, let $\lambda_n<\delta$ be such that there is a $B_n\subset\mathbb{R}^{\mathcal{V}[g\uphar\lambda_n]}$ such that $B_n$ is ${<}\delta$-uB in $\mathcal{V}[g\uphar\lambda_n]$ and $A_n=(B_n)^*_{g\uphar\lambda_n}$.
Let $\lambda=\sup_{n<\omega}\lambda_n$.
Since $\delta$ is regular, $\lambda<\delta$.
Let $\delta'<\delta$ be the least Woodin cardinal above $\lambda$ in $\mathcal{V}$.
Then by \cite[Fact 3.3]{dm_of_self_it}, all $A_n$'s are projective in $\mathrm{Code}(\Sigma^g_{\mathcal{P}\vert\delta'})$.
It follows, however, that even if $\delta'<\xi<\delta$, $\mathrm{Code}(\Sigma^g_{\mathcal{P}\vert\xi})$ is projective in $\mathrm{Code}(\Sigma^g_{\mathcal{P}\vert\delta'})$, which contradicts \cite[Lemma 3.4]{dm_of_self_it}.\footnote{The anonymous referee pointed out that, in this proof, we could use the fact that $\{\mathrm{Code}(\Sigma^g_{\mathcal{P}\vert\xi})\mid\xi<\delta\}$ is Wadge cofinal in $\Gamma^*_g$, as already mentioned in the paragraph following \cref{stabilizing_parameters}.}

Now we can easily show that $\mathsf{DC}$ holds in $\mathsf{CDM}^+(\mathcal{Q}, \eta)$.
In \cite{So21}, Solovay showed that $\mathsf{AD}+\mathsf{DC}_{\mathbb{R}}+\cf(\Theta)>\omega$ implies that $\mathsf{DC}_{\power(\mathbb{R})}$.
Since $\mathsf{CDM}^+(\mathcal{Q}, \eta)\models\mathsf{AD}^+$ by \cref{determinacy_in_CDM_plus_1}, $\mathsf{CDM}^+(\mathcal{Q}, \eta)\models\mathsf{DC}_{\power(\mathbb{R})}$.
Then in $\mathsf{CDM}^+(\mathcal{Q}, \eta)$, $\mathsf{DC}$ reduces to $\mathsf{DC}_X$ where $X=\bigcup_{\xi<\delta^{\mathcal{Q}, \eta}_{\infty}}{}^{\omega}\xi$, because any element of $\mathsf{CDM}^+(\mathcal{Q}, \eta)$ is ordinal definable in parameters from $X$ and sets of reals.
Since $\cf(\delta^{\mathcal{Q}, \eta}_{\infty})=\cf(\delta)>\omega$ in $\mathcal{V}[g]$, any $\omega$-sequence from $X$ can be easily coded into an element of $X$.
By this observation, $\mathsf{DC}_X$ in $\mathcal{V}[g]$ implies $\mathsf{DC}_X$ in $\mathsf{CDM}^+(\mathcal{Q}, \eta)$.
Therefore, $\mathsf{CDM}^+(\mathcal{Q}, \eta)\models\mathsf{DC}$.

The regularity of $\Theta$ in $\mathsf{CDM}^+(\mathcal{Q}, \eta)$ follows from the proof of \cite[Theorem 1.3]{dm_of_self_it}.
Let $\Theta=\Theta^{\mathsf{CDM}^+(\mathcal{Q}, \eta)}$.
Suppose toward a contradiction that there is a cofinal map $f\colon\mathbb{R}^*_g\to\Theta$ in $\mathsf{CDM}^{+}(\mathcal{Q}, \eta)$.
Then there are a formula $\phi$ in the language for $\mathsf{CDM}^+(\mathcal{Q}, \eta)$, an ordinal $\gamma$, $Y\in {}^{\omega}\xi$ for some $\xi<\delta^{\mathcal{Q}, \eta}_{\infty}$, $Z\in\Gamma^*_g$, $x\in\mathbb{R}^*_g$ and $\vec{\beta}\in{}^{<\omega}\gamma$ such that
\[
f=\{\langle u, \zeta\rangle\in\mathbb{R}^*_g\times\Theta\mid(\mathsf{CDM}^+(\mathcal{Q}, \eta)\vert\gamma; {\in}, \vec{\mu})\models\phi[u, \zeta, Y, Z, x, \vec{\beta}]\}.
\]
We take a genericity iterate $\mathcal{R}$ of $\mathcal{Q}$ above $\eta$ such that $\{\vec{\beta}, \eta\}\cup\mathrm{ran}(Y)\subset\mathrm{ran}(\pi^{\mathcal{R}, \eta}_{\mathcal{R}, \infty})$ and a maximal $\mathcal{R}$-generic $k\subset\mathrm{Col}(\omega, {<}\delta)$.
Let $\xi_Y<\delta$ be such that $\mathrm{ran}(Y)\subset\pi^{\mathcal{R}, \eta}_{\mathcal{R}, \infty}[\xi_Y]$.
Let $y\in\mathbb{R}^*_{k}$ code a function $f_y\colon\omega\to\xi_Y$ such that for any $i\in\omega$, $Y(i)=\pi^{\mathcal{R}, \eta}_{\mathcal{R}, \infty}(f_y(i))$.
Also, we may assume that $Z=\mathrm{Code}(\Sigma^g_{\mathcal{P}\vert\xi_{Z}})$ for some $\xi_Z<\delta$.
Let $z\in\mathbb{R}^*_{k}$ code $\pi_{\mathcal{P}, \mathcal{R}}\uphar(\mathcal{P}\vert\xi_Z)$.
Let $\eta'\in[\max\{\eta, \xi_Y, \pi_{\mathcal{P}, \mathcal{R}}(\xi_Z)\}, \delta)$ be such that $x, y, z\in\mathcal{R}[k\uphar\eta']$.
Because $\mathsf{CDM}(\mathcal{Q}, \eta)=\mathsf{CDM}(\mathcal{R}, \eta)^{\mathcal{V}_{\mathcal{R}}[h]}$,
\[
f=\{\langle u, \zeta\rangle\in\mathbb{R}^*_g\times\Theta\mid\mathcal{V}_{\mathcal{R}}[x, y, z][u]\models\phi^*(u, \zeta, x, y, z, \eta, \delta, \vec{\beta}, \gamma)\},
\]
where $\phi^*$ is obtained from $\phi$ as in the proof of \cref{club_dichotomy}.

Let $\delta'<\delta$ be the least Woodin cardinal of $\mathcal{R}$ above $\eta'$ and let $\eta''\in(\delta', \delta)$ be an inaccessible cardinal of $\mathcal{R}$ such that
\[
(\mathsf{CDM}^+(\mathcal{R}, \eta); {\in}, \vec{\mu})\models w(\mathrm{Code}(\Sigma^k_{\mathcal{R}\vert\eta''}))>\sup f[\mathbb{R}^{k\uphar\delta'}],
\]
where $w(\text{--})$ denotes the Wadge rank of a set of reals.
Such an $\eta''$ exists because $\mathbb{R}^{k\uphar\delta'}$ is countable in $\mathsf{CDM}^+(\mathcal{R}, \eta)$, $\cf(\Theta)>\omega$ holds in $\mathsf{CDM}^+(\mathcal{Q}, \eta)=\mathsf{CDM}^+(\mathcal{R}, \eta)$, and $\{\mathrm{Code}(\mathcal{R}^k_{\mathcal{R}\vert\xi})\mid\xi<\delta\}$ is Wadge cofinal in $\Gamma^*_g=\Gamma^*_k$.
Since $f$ is cofinal, there is an $r\in\mathbb{R}^*_k$ such that $f(r)>w(\mathrm{Code}(\Sigma^k_{\mathcal{R}\vert\delta''}))$, where $\delta''<\delta$ is a sufficiently large Woodin cardinal of $\mathcal{R}$ above $\eta''$ such that $\mathrm{Code}(\Sigma^g_{\mathcal{R}\vert\delta''})$ is not projective in $\mathrm{Code}(\Sigma^g_{\mathcal{R}\vert\eta''})$.\footnote{Actually, one can choose $\delta''$ as the least Woodin cardinal of $\mathcal{R}$ above $\eta'$, see \cite[Lemma 3.4]{dm_of_self_it}.}

Using the extender algebra at $\delta'$, we can take an $\mathcal{R}^*\in\mathcal{F}^*_g(\mathcal{R}, \eta')$ and an $\mathcal{R}^*$-generic $k^*\subset\mathrm{Col}(\omega, {<}\delta)$ such that $k\uphar\eta'\subset k^*$ and $r\in\mathcal{R}^*[k^*\uphar\pi_{\mathcal{R},\mathcal{R}^*}(\delta')]$.
Then let $\mathcal{S}$ be a non-dropping iterate of $\mathcal{R}^*$ such that it is genericity iterate of $\mathcal{R}$ and $\crit(\pi_{\mathcal{R}^*, \mathcal{S}})>\pi_{\mathcal{R},\mathcal{R}^*}(\delta')$.
Let $l\subset\mathrm{Col}(\omega, {<}\delta)$ be a maximal $\mathcal{S}$-generic such that $k^*\uphar\pi_{\mathcal{R}, \mathcal{R}^*}(\delta')\subset l$.

Let $\pi^{+}_{\mathcal{R}, \mathcal{S}}\colon\mathcal{R}[k\uphar\eta']\to\mathcal{S}[k\uphar\eta']$ be the canonical liftup of $\pi_{\mathcal{R}, \mathcal{S}}$.
Mainly because $\mathsf{CDM}^+(\mathcal{R}, \eta)=\mathsf{CDM}^+(\mathcal{S}, \eta)$, we have $f = \pi^+_{\mathcal{R}, \mathcal{S}}(f)$ as in the proof of \cref{club_dichotomy}.
Then the elementarity of $\pi^+_{\mathcal{R}, \mathcal{S}}$ implies that
\[
(\mathsf{CDM}^+(\mathcal{S}, \eta), {\in}, \vec{\mu})\models w(\mathrm{Code}(\Sigma^l_{\mathcal{S}\vert\pi_{\mathcal{R}, \mathcal{S}}(\eta'')}))>\sup f[\mathbb{R}^{l\uphar\pi_{\mathcal{R}, \mathcal{S}}(\delta')}].
\]
Since $\mathbb{R}^{l\uphar\pi_{\mathcal{R, \mathcal{S}}}(\delta')}\supseteq\mathbb{R}^{k^*\uphar\pi_{\mathcal{R}, \mathcal{R}^*}(\delta')}\ni r$, it follows that $w(\mathrm{Code}(\Sigma^l_{\mathcal{S}\vert\pi_{\mathcal{R}, \mathcal{S}}(\eta'')}))>f(r)$.
Also, as $\Sigma^l_{\mathcal{S}\vert\pi_{\mathcal{R}, \mathcal{S}}(\eta'')}$ is a tail strategy of $\Sigma^k_{\mathcal{R}\vert\eta''}$, $\mathrm{Code}(\Sigma^l_{\mathcal{S}\vert\pi_{\mathcal{R}, \mathcal{S}}(\eta'')})$ is projective in $\mathrm{Code}(\Sigma^k_{\mathcal{R}\vert\eta''})$.
Then by the choice of $\delta''$,
\[
w(\mathrm{Code}(\Sigma^k_{\mathcal{R}\vert\delta''}))> w(\mathrm{Code}(\Sigma^l_{\mathcal{S}\vert\pi_{\mathcal{R}, \mathcal{S}}(\eta'')})).
\]
Therefore, $w(\mathrm{Code}(\Sigma^k_{\mathcal{R}\vert\delta''}))>f(r)$, which contradicts the choice of $r$.
\end{proof}

\subsection{The measurability of $\Theta$}

A filter $\mu$ is called \emph{$\mathbb{R}$-complete} if for any function $f\colon\mathbb{R}\to\mu$, $\bigcap_{x\in\mathbb{R}}f(x)\in\mu$.
In the context of $\mathsf{AD}$, we say that \emph{$\Theta$ is measurable} if there is an $\mathbb{R}$-complete normal ultrafilter on $\Theta$.
To show that $\Theta$ is measurable, we need to suppose that $\delta^{\mathcal{Q}, \eta}_{\infty}>\Theta$.
The consistency of this assumption from large cardinals in $\mathcal{V}$ will be shown in \cref{main_thm_2} below.

\begin{thm}\label{Theta_measurable}
    Let $\mathcal{Q}\in I^*_g(\mathcal{P}, \Sigma)$ and $\eta<\delta$ be such that $\mathcal{Q}$ is a genericity iteration of $\mathcal{P}$ and $(\mathcal{Q}, \eta)$ stabilizes $\delta_{\infty}$.
    Also, suppose that $\delta^{\mathcal{Q}, \eta}_{\infty}>\Theta$.
    Then the restriction of the club filter (in $\mathcal{V}[g]$) on $\Theta\cap\mathrm{Cof}(\omega):=\{\alpha<\Theta\mid\cf(\alpha)=\omega\}$ to $\mathsf{CDM}^+(\mathcal{Q}, \eta)$ is an $\mathbb{R}$-complete normal ultrafilter in $\mathsf{CDM}^+(\mathcal{Q}, \eta)$.
    Therefore,
    \[
    \mathsf{CDM}^+(\mathcal{Q}, \eta)\models\Theta\text{ is measurable.}
    \]
\end{thm}
\begin{proof}
    We write $\Theta=\Theta^{\mathsf{CDM}^+(\mathcal{Q}, \eta)}$.
    Let $\nu$ be the club filter on $\Theta\cap\mathrm{Cof}(\omega)$.
    Repeating the same argument in the proof of \cref{locality_of_CDM_plus}, we will show that $\nu\cap\mathsf{CDM}^+(\mathcal{Q}, \eta)\in\mathsf{CDM}^+(\mathcal{Q}, \eta)$ and it is an $\mathbb{R}$-complete normal ultrafilter in $\mathsf{CDM}^+(\mathcal{Q}, \eta)$.
    
    Let $\mathcal{R}$ be a genericity iterate of $\mathcal{Q}$, let $k\subset\mathrm{Col}(\omega, {<}\delta)$ be an $\mathcal{R}$-maximal generic, and let $\xi<\delta$.
    For any $\mathcal{R}^*\in\mathcal{F}^*_k(\mathcal{R}, \xi)$, we set
    \[
    \sigma'_{\mathcal{R}^*, \xi}=\sup(\mathrm{ran}(\pi^{\mathcal{R}, \xi}_{\mathcal{R}^*, \infty})\cap\Theta)\in\Theta\cap\mathrm{Cof}(\omega)
    \]
    and define
    \[
    C'_{\mathcal{R}, \xi}=\{\sigma'_{\mathcal{R}^*, \xi}\mid\mathcal{R}^*\in\mathcal{F}^*_k(\mathcal{R}, \xi)\land\Theta\in\mathrm{ran}(\pi^{\mathcal{R}, \xi}_{\mathcal{R}^*, \infty})\}.
    \]
    By the same argument as \cref{finding_clubs,club_dichotomy}, we have the following claims.
    
    \begin{claim}
        Whenever $\mathcal{R}$ is a genericity iterate of $\mathcal{Q}$ above $\eta$, $k\subset\mathrm{Col}(\omega, {<}\delta)$ is a maximal $\mathcal{R}$-generic, and $\xi\in[\eta, \delta)$, then the set $C'_{\mathcal{R}, \xi}$ contains a club subset of $\Theta$ in $\mathcal{V}_{\mathcal{R}}[k]$.
    \end{claim}
    
    \begin{claim}
        Let $A\subset\Theta$ be in $\mathsf{CDM}^+(\mathcal{Q}, \eta)$.
        Then there are a genericity iterate $\mathcal{R}$ of $\mathcal{Q}$ and a $\xi\in [\eta, \delta)$ such that the following hold:
        \begin{enumerate}
            \item If $\sigma'_{\mathcal{R}, \xi}\in A$ then $C'_{\mathcal{R}, \xi}\subset A$.
            \item If $\sigma'_{\mathcal{R}, \xi}\notin A$ then $C'_{\mathcal{R}, \xi}\subset (\Theta\cap\mathrm{Cof}(\omega))\setminus A$.
        \end{enumerate}
    \end{claim}

    These claims imply that $\nu\cap\mathsf{CDM}^+(\mathcal{Q}, \eta)$ is an ultrafilter on $\Theta\cap\mathrm{Cof}(\omega)$ over $\mathsf{CDM}^+(\mathcal{Q}, \eta)$.
    Another consequence of the claims is that
    \[
    A\in\nu\iff\exists B\in\mu_{\Theta}(\sup[B]\subset A),
    \]
    where $\mu_{\Theta}$ is the club filter on $\power_{\omega_1}(\Theta)$ in $\mathcal{V}[g]$ and $\sup[B]=\{\sup(\sigma)\mid\sigma\in B\}$.
    As $\mu_{\Theta}\cap\mathsf{CDM}^+(\mathcal{Q}, \eta)\in\mathsf{CDM}^+(\mathcal{Q}, \eta)$, we have $\nu\cap\mathsf{CDM}^+(\mathcal{Q}, \eta)\in\mathsf{CDM}^+(\mathcal{Q}, \eta)$.
    
    It is easy to see that in $\mathsf{CDM}^+(\mathcal{Q}, \eta)$, $\nu':=\nu\cap\mathsf{CDM}^+(\mathcal{Q}, \eta)$ is a $\Theta$-complete normal ultrafilter, since $\nu$ is a $\Theta$-complete normal filter in $\mathcal{V}[g]$.
    It remains to show the $\mathbb{R}$-completeness of $\nu'$.
    This follows from $\Theta$-completeness by the proof of \cite[Theorem 2.6]{PmaxSquare}.
    We write their short proof here for the reader's convenience.
    From now, we work in $\mathsf{CDM}^+(\mathcal{Q}, \eta)$ and let $f\colon\mathbb{R}\to\nu'$.
    For each $\alpha<\Theta$, let 
    \[
    B_{\alpha}=\{x\in\mathbb{R}\mid\alpha\in f(x)\}.
    \]
    Since $\mathsf{AD}_{\mathbb{R}}$ holds (in $\mathsf{CDM}^+(\mathcal{Q}, \eta)$) by \cref{determinacy_in_CDM_plus_1}, there is no Wadge-cofinal function from an ordinal to $\power(\mathbb{R})$ by \cite[Remark 2.5]{PmaxSquare}.
    So we can take some set of reals $B^*\subset\mathbb{R}$ with Wadge rank $\geq\sup\{B_{\alpha}\mid\alpha<\Theta\}$.
    Then $\lvert\{B_{\alpha}\mid\alpha<\Theta\}\rvert<\Theta$ because otherwise we could define a prewellordering of reals $\prec$ of length $\Theta$ as follows: for any $x_0, x_1\in\mathbb{R}$, $x_0\prec x_1$ if and only if for any $i\in\{0, 1\}$, $x_i$ codes a continuous function $f_i$ such that $B_{\alpha_i}=f_i^{-1}[B^*]$ for some ordinal $\alpha_i<\Theta$, and $\alpha_0<\alpha_1$.
    Since $\nu'$ is $\Theta$-complete, we can choose a $B\subset\mathbb{R}$ such that $\{\alpha<\Theta\mid B_{\alpha}=B\}\in\nu'$.
    For any $x\in\mathbb{R}$, $f(x)\in\nu'$.
    So there is an $\alpha<\Theta$ such that $B_{\alpha}=B$ and $\alpha\in f(x)$, which implies $x\in B_{\alpha}=B$.
    Therefore, $B=\mathbb{R}$.
    Then $\bigcap_{x\in\mathbb{R}}f(x)=\{\alpha<\Theta\mid B_{\alpha}=\mathbb{R}\}\in\nu'$, which completes the proof of $\mathbb{R}$-completeness.
\end{proof}

\begin{cor}\label{P(R)-supercompactness}
    Let $\mathcal{Q}\in I^*_g(\mathcal{P}, \Sigma)$ and $\eta<\delta$ be such that $\mathcal{Q}$ is a genericity iteration of $\mathcal{P}$ and $(\mathcal{Q}, \eta)$ stabilizes $\delta_{\infty}$.
    Also, suppose that $\delta^{\mathcal{Q}, \eta}_{\infty}>\Theta$.
    Then, in $\mathsf{CDM}^+(\mathcal{Q}, \eta)$, $\omega_1\text{ is }\power(\mathbb{R})$-supercompact witnessed by the club filter.
\end{cor}
\begin{proof}
    It is a folklore result that if $\mathsf{AD}_{\mathbb{R}}+\Theta$ is measurable, then $\omega_1$ is $\power(\mathbb{R})$-supercompact.
    This is proved in \cite[Theorem 3.1]{Tr15MALQ}.
    In the proof, the supercompact measure $\mu_{\power(\mathbb{R})}$ on $\power_{\omega_1}(\power(\mathbb{R}))$ is defined as follows: Let $\Gamma_{\alpha}\subset\power(\mathbb{R})$ be the set of sets of reals with Wadge rank $<\alpha$.
    Then for any countable $A\subset\power(\mathbb{R})$, we define
    \[
    A\in\mu_{\power(\mathbb{R})}\iff\{\alpha<\Theta\mid A\cap\power_{\omega_1}(\Gamma_{\alpha})\in\mu_{\Gamma_{\alpha}}\}\in\nu,
    \]
    where $\nu$ is an $\mathbb{R}$-complete normal ultrafilter on $\Theta$ and $\mu_{\Gamma_{\alpha}}$ is the club filter on $\power_{\omega_1}(\Gamma_{\alpha})$, which is an ultrafilter by \cite{Wo20}.
    Therefore, $\mu_{\power(\mathbb{R})}$ is the club filter on $\power_{\omega_1}(\power(\mathbb{R}))$.
\end{proof}

\begin{ques}
    Is there a variant of $\mathsf{CDM}$ where $\omega_1$ is $\power(\power(\mathbb{R}))$-supercompact?
\end{ques}

\begin{ques}
    In $\mathsf{CDM}^+(\mathcal{Q}, \eta)$, are all normal ultrafilters on $\power_{\omega_1}(\alpha), \power_{\omega_1}(\power(\mathbb{R}))$, and $\Theta$ club filters?
\end{ques}

\section{Value of $\delta_{\infty}$}

Let $\eta<\delta$ and let $\mathcal{Q}$ be a genericity iterate of $\mathcal{P}$.
First, we will argue which cardinal of $\mathcal{Q}$ is moved to $\Theta^{\mathsf{CDM}^+(\mathcal{Q}, \eta)}$ under the direct limit map $\pi^{\mathcal{Q}, \eta}_{\mathcal{Q}, \infty}\colon\mathcal{Q}\to\mathcal{M}_{\infty}(\mathcal{Q}, \eta)$.
Let
\[
\kappa^{\mathcal{Q}, \eta}=\begin{cases}
\text{the least ${<}\delta$-strong cardinal in $\mathcal{Q}$ in the interval $(\eta, \delta)$} & \text{if it exists,}\\
\delta & \text{otherwise.}
\end{cases}
\]
For $\kappa>\eta$, we say that $\kappa$ is an $\eta$-cutpoint of an lbr hod premouse $\mathcal{M}$ if for any extender $E$ on the extender sequence of $\mathcal{M}$, if $\crit(E)<\kappa\leq\lh(E)$ then $\crit(E)<\eta$.
Then $\kappa^{\mathcal{Q}, \eta}$ is the largest $\eta$-cutpoint of $\mathcal{Q}$ less than or equal to $\delta$.
Also, we let $\kappa^{\mathcal{Q}, \eta}_{\infty}$ be the direct limit image of $\kappa^{\mathcal{Q}, \eta}$ in $\mathcal{M}_{\infty}(\mathcal{Q}, \eta)$.

\begin{thm}\label{value_of_Theta}
Let $\eta<\delta$ and let $\mathcal{Q}\in I^*_g(\mathcal{P}, \Sigma)$ be a genericity iterate of $\mathcal{P}$.
Then $\kappa^{\mathcal{Q}, \eta}_{\infty}=\Theta^{\mathsf{CDM}^{+}(\mathcal{Q}, \eta)}$.
\end{thm}
\begin{proof}
    Let $\Theta=\Theta^{\mathsf{CDM}^{+}(\mathcal{Q}, \eta)}$.
    The following claim implies that $\Theta\leq\kappa^{\mathcal{Q}, \eta}_{\infty}$.
    \begin{claim}
        $\Theta$ is an $\eta$-cutpoint of $\mathcal{M}_{\infty}(\mathcal{Q}, \eta)$.
    \end{claim}
    \begin{proof}
        The claim follows from the proof of \cite[Theorem 1.7]{char_of_ext_in_HOD}, but we write it here for the reader's convenience.
        We work in $\mathsf{CDM}^{+}(\mathcal{Q}, \eta)$.
        Then $\mathsf{AD}^+ +\mathsf{AD}_{\mathbb{R}}$ holds, so $\Theta$ is a limit member of the Solovay sequence.
        Suppose toward a contradiction that there is an extender $E$ in the extender sequence of $\mathcal{M}_{\infty}(\mathcal{Q}, \eta)$ such that $\eta\leq\crit(E)<\Theta\leq\lh(E)$. Let $\kappa=\crit(E)$ and let $\theta_{\alpha+1}<\Theta$ be the least member of the Solovay sequence above $\kappa$. By \cite[Theorem 1.5]{char_of_ext_in_HOD}\footnote{The theorem is not stated in \cite{char_of_ext_in_HOD} in the generality we need. See \cite[Theorem 0.3]{Suslin_and_cutpoints}.}, there is a countably complete ultrafilter $U$ (over $\mathsf{CDM}^+(\mathcal{Q}, \eta)$) such that $\kappa=\crit(U)$ and $\pi_{U}(\kappa)\geq\pi_{E}(\kappa)$.
        Then  $U$ is ordinal definable by Kunen's theorem (\cite[Theorem 7.6]{St09}).\footnote{Some literature assumes $\mathsf{AD}+\mathsf{DC}$ for Kunen's theorem, but $\mathsf{AD}+\mathsf{DC}_{\mathbb{R}}$ is enough.}
        So there is an OD surjection $\power(\kappa)\to\pi_{U}(\kappa)$.
        Since $\theta_{\alpha+1}<\pi_{E}(\kappa)\leq\pi_{U}(\kappa)$, we can take an OD surjection $f\colon\power(\kappa)\to\theta_{\alpha+1}$.
        Let $A$ be any set of reals of Wadge rank $\theta_{\alpha}$.
        Then there is an $\mathrm{OD}(A)$ surjection $\mathbb{R}\to\kappa$ as $\kappa<\theta_{\alpha+1}$.
        Moschovakis coding lemma (\cite[Section 7D]{Mo09}) implies that there is an $\mathrm{OD}(A)$ surjection $g\colon\mathbb{R}\to\power(\kappa)$.
        Then $f\circ g\colon\mathbb{R}\to\theta_{\alpha+1}$ is an $\mathrm{OD}(A)$ surjection, which is a contradiction.
    \end{proof}
    Suppose toward a contradiction that $\Theta<\kappa^{\mathcal{Q}, \eta}_{\infty}$.
    Then there is an $\mathcal{R}\in\mathcal{F}^*_g(\mathcal{Q}, \eta)$ such that $\Theta=\pi^{\mathcal{Q}, \eta}_{\mathcal{R}, \infty}(\xi)$ for some $\xi<\kappa^{\mathcal{R}, \eta}$.
    Since every extender in the extender sequence of $\mathcal{R}$ overlapping $\kappa^{\mathcal{R}, \eta}$ has critical point $\leq\eta$, $\pi^{\mathcal{Q}, \eta}_{\mathcal{R}, \infty}\uphar\kappa^{\mathcal{R}, \eta}$ is an iteration map according to the fragment of $\Sigma_{\mathcal{R}}$ acting the iterations based on the window $(\eta, \kappa^{\mathcal{R}, \eta})$.
    Since $\mathrm{Code}(\Sigma^g_{\mathcal{R}\vert\kappa^{\mathcal{R}, \eta}})\in\Gamma^*_g$, $\Theta=\pi^{\mathcal{Q}, \eta}_{\mathcal{R}, \infty}(\xi)$ is ordinal definable from a set of reals in $\mathsf{CDM}^+(\mathcal{Q}, \eta)$, which is a contradiction.
    Therefore, $\Theta=\kappa^{\mathcal{Q}, \eta}_{\infty}$.
\end{proof}

The following is an immediate corollary of \cref{value_of_Theta}.

\begin{cor}\label{main_thm_2}
Let $\eta<\delta$ and let $\mathcal{Q}$ be a genericity iterate of $\mathcal{P}$.
If, in $\mathcal{V}$, $\delta$ is a limit of Woodin cardinals that is also a limit of ${<}\delta$-strong cardinals, then $\delta^{\mathcal{Q}, \eta}_{\infty}>\Theta$.
\end{cor}

Next, we give a sufficient condition for an ordinal in $[\delta, \delta^{\mathcal{Q}, \eta}_{\infty}]$ to be a (regular) cardinal in $\mathsf{CDM}^+(\mathcal{Q}, \eta)$.

\begin{thm}
Let $\eta<\delta$ and let $\mathcal{Q}$ be a genericity iterate of $\mathcal{P}$ such that $(\mathcal{Q}, \eta)$ stabilizes $\delta_{\infty}$.
Also, let $\lambda\in[\delta, \delta^{\mathcal{Q}, \eta}_{\infty}]$.
Then the following hold.
\begin{enumerate}
\item Suppose that for any $\eta'\in [\eta, \delta)$ and any genericity iteration $\mathcal{S}$ of $\mathcal{Q}$ above $\eta$, $\lambda$ is a cardinal in $\mathcal{M}_{\infty}(\mathcal{S}, \eta')$.
Then $\lambda$ is a cardinal in $\mathsf{CDM}^{+}(\mathcal{Q}, \eta)$.
\item Suppose that for any $\eta'\in [\eta, \delta)$ and any genericity iteration $\mathcal{S}$ of $\mathcal{Q}$ above $\eta$, $\lambda$ is a regular cardinal in $\mathcal{M}_{\infty}(\mathcal{S}, \eta')$.
Then $\lambda$ is a regular cardinal in $\mathsf{CDM}^{+}(\mathcal{Q}, \eta)$.
\end{enumerate}
\end{thm}
\begin{proof}
We only give the proof of (1) here because the same argument shows (2) as well.
Suppose that $\nu<\lambda$ and that there is a surjection $f\colon\nu\to\lambda$ in $\mathsf{CDM}^{+}(\mathcal{Q}, \eta)$.
Then there are a formula $\phi$ in the language for $\mathsf{CDM}^+(\mathcal{Q}, \eta)$, an ordinal $\gamma$, $Y\in {}^{\omega}\xi$ for some $\xi<\delta^{\mathcal{Q}, \eta}_{\infty}$, $Z\in\Gamma^*_g$, $x\in\mathbb{R}^*_g$, $\vec{\beta}\in{}^{\omega}\gamma$ such that
\[
f=\{\langle\alpha, \beta\rangle\in\nu\times\lambda\mid(\mathsf{CDM}^+(\mathcal{Q}, \eta)\vert\gamma; {\in}, \vec{\mu})\models\phi[\alpha, \beta, Y, Z, x, \vec{\beta}]\}
\]
We take a genericity iterate $\mathcal{R}$ of $\mathcal{Q}$ above $\eta$ such that $\{\vec{\beta}, \gamma, \nu, \lambda\}\cup\mathrm{ran}(Y)\subset\pi^{\mathcal{R}, \eta}_{\mathcal{R}, \infty}[\delta]$ and a maximal $\mathcal{R}$-generic $k\subset\mathrm{Col}(\omega, {<}\delta)$.
Let $\xi_Y<\delta$ be such that $\mathrm{ran}(Y)\subset\pi^{\mathcal{R}, \eta}_{\mathcal{R}, \infty}[\xi_Y]$.
Let $y\in\mathbb{R}^*_{k}$ code a function $f_y\colon\omega\to\xi_Y$ such that for any $i\in\omega$, $Y(i)=\pi^{\mathcal{R}, \eta}_{\mathcal{R}, \infty}(f_y(i))$.
Also, we may assume that $Z=\mathrm{Code}(\Sigma^g_{\mathcal{P}\vert\xi_{Z}})$ for some $\xi_Z<\delta$.
Let $z\in\mathbb{R}^*_{k}$ code $\pi_{\mathcal{P}, \mathcal{R}}\uphar(\mathcal{P}\vert\xi_Z)$.
Then fix any $\eta'\in[\max\{\eta, \xi_Y, \pi_{\mathcal{P}, \mathcal{R}}(\xi_Z)\}, \delta)$ such that $x, y, z\in\mathcal{R}[k\uphar\eta']$.
Because $\mathsf{CDM}^+(\mathcal{Q}, \eta)=\mathsf{CDM}^+(\mathcal{R}, \eta)^{\mathcal{V}_{\mathcal{R}}[h]}$,
\[
f=\{\langle\alpha, \beta\rangle\in\nu\times\lambda\mid\mathcal{V}_{\mathcal{R}}[x, y, z]\models\phi^*[\alpha, \beta, x, y, z, \eta, \delta, \vec{\beta}, \gamma]\},
\]
where $\phi^*$ is obtained from $\phi$ as in the proof of \cref{club_dichotomy}.

Also, let $\mathcal{S}$ be a genericity iterate of $\mathcal{R}$ above $\eta'$ such that $\nu, \lambda\in\mathrm{ran}(\pi^{\mathcal{S}, \eta'}_{\mathcal{S}, \infty})$.
We can take such an $\mathcal{S}$ because $\delta^{\mathcal{Q}, \eta}_{\infty}=\delta^{\mathcal{S}, \eta'}_{\infty}$ as $(\mathcal{Q}, \eta)$ stabilizes $\delta_{\infty}$.
Let $\nu_{\mathcal{S}}$ and $\lambda_{\mathcal{S}}$ in $\mathcal{S}$ be the preimages of $\nu$ and $\lambda$ under $\pi^{\mathcal{S}, \eta'}_{\mathcal{S}, \infty}$ respectively.
Then $\lambda_{\mathcal{S}}>\eta'$ since otherwise $\lambda=\pi^{\mathcal{S}, \eta'}_{\mathcal{S}, \infty}(\lambda_{\mathcal{S}})=\lambda_{\mathcal{S}}\leq\eta'<\delta$, which contradicts $\lambda\geq\delta$.
Also, because $\lambda$ is a cardinal in $\mathcal{M}_{\infty}(\mathcal{S}, \eta')$ by the assumption on $\lambda$, $\lambda_{\mathcal{S}}$ is a cardinal in $\mathcal{S}$.

In $\mathcal{S}[x, y, z]$, we define a partial function $\overline{f}\colon\nu_{\mathcal{S}}\rightharpoonup\lambda_{\mathcal{S}}$ by
\[
\alpha\in\dom(\overline{f})\land\overline{f}(\alpha)=\beta\iff f(\pi^{\mathcal{S}, \eta'}_{\mathcal{S}, \infty}(\alpha))=\pi^{\mathcal{S}, \eta'}_{\mathcal{S}, \infty}(\beta)
\]
for any $\alpha<\nu_{\mathcal{S}}$ and $\beta<\lambda_{\mathcal{S}}$.
We will show that $\overline{f}$ is surjective, which contradicts the fact that $\lambda_{\mathcal{S}}$ is a cardinal in $\mathcal{S}$.
Now let $\beta^*<\lambda_{\mathcal{S}}$.
Let $\alpha^*<\nu$ be such that
\[
f(\alpha^*)=\pi^{\mathcal{S}, \eta'}_{\mathcal{S}, \infty}(\beta^*).
\]
Take a genericity iterate $\mathcal{W}$ of $\mathcal{S}$ above $\eta'$ such that $\alpha^*\in\mathrm{ran}(\pi^{\mathcal{W}, \eta'}_{\mathcal{W}, \infty})$.
Let $\alpha^*_{\mathcal{W}}$ be the preimage of $\alpha^*$ under $\pi^{\mathcal{W}, \eta'}_{\mathcal{W}, \infty}$ in $\mathcal{W}$.
Then
\[
f(\pi^{\mathcal{W}, \eta'}_{\mathcal{W}, \infty}(\alpha^*_{\mathcal{W}}))=\pi^{\mathcal{W}, \eta'}_{\mathcal{W}, \infty}(\pi^+_{\mathcal{V}_{\mathcal{S}}, \mathcal{V}_{\mathcal{W}}}(\beta^*)),
\]
where $\pi^+_{\mathcal{V}_{\mathcal{S}}, \mathcal{V}_{\mathcal{W}}}\colon\mathcal{V}_{\mathcal{S}}[x, y, z]\to\mathcal{V}_{\mathcal{W}}[x, y, z]$ is the canonical liftup of $\pi_{\mathcal{V}_{\mathcal{S}}, \mathcal{V}_{\mathcal{W}}}$.
Mainly because $\mathsf{CDM}^+(\mathcal{S}, \eta)=\mathsf{CDM}^+(\mathcal{W}, \eta)$, we have $f=\pi^+_{\mathcal{V}_{\mathcal{S}}, \mathcal{V}_{\mathcal{W}}}(f)$ as in the proof of \cref{club_dichotomy}. (Also, see the proof of \cref{determinacy_in_CDM_plus_2}.)
It follows that
\[
(*)\;\;\; \pi^+_{\mathcal{V}_{\mathcal{S}}, \mathcal{V}_{\mathcal{W}}}(f)(\pi^{\mathcal{W}, \eta'}_{\mathcal{W}, \infty}(\alpha^*_{\mathcal{W}}))=\pi^{\mathcal{W}, \eta'}_{\mathcal{W}, \infty}(\pi^+_{\mathcal{V}_{\mathcal{S}}, \mathcal{V}_{\mathcal{W}}}(\beta^*)).
\]
On the other hand, by the definition of $\overline{f}$ and the elementarity of $\pi^+_{\mathcal{V}_{\mathcal{S}}, \mathcal{V}_{\mathcal{W}}}$, we have that for any $\alpha<\nu_{\mathcal{W}}$ and any $\beta<\lambda_{\mathcal{W}}$,
\begin{multline*}
(**)\;\;\; \alpha\in\dom(\pi^+_{\mathcal{V}_{\mathcal{S}}, \mathcal{V}_{\mathcal{W}}}(\overline{f}))\land\pi^+_{\mathcal{V}_{\mathcal{S}}, \mathcal{V}_{\mathcal{W}}}(\overline{f})(\alpha)=\beta\\
\iff \pi^+_{\mathcal{V}_{\mathcal{S}}, \mathcal{V}_{\mathcal{W}}}(f)(\pi^{\mathcal{W}, \eta'}_{\mathcal{W}, \infty}(\alpha))=\pi^{\mathcal{W}, \eta'}_{\mathcal{W}, \infty}(\beta).
\end{multline*}
It follows from $(*)$ and $(**)$ that
\[
\pi^{+}_{\mathcal{V}_{\mathcal{S}}, \mathcal{V}_{\mathcal{W}}}(\overline{f})(\alpha^*_{\mathcal{W}})=\pi^+_{\mathcal{V}_{\mathcal{S}}, \mathcal{V}_{\mathcal{W}}}(\beta^*).
\]
Thus, $\pi^+_{\mathcal{V}_{\mathcal{S}}, \mathcal{V}_{\mathcal{W}}}(\beta^*)\in\mathrm{ran}(\pi^{+}_{\mathcal{V}_{\mathcal{S}}, \mathcal{V}_{\mathcal{W}}}(\overline{f}))$.
By the elementarity of $\pi^{+}_{\mathcal{V}_{\mathcal{S}}, \mathcal{V}_{\mathcal{W}}}$, $\beta^*\in\mathrm{ran}(\overline{f})$.
\end{proof}

\begin{cor}\label{main_thm_3}
$\delta^{\mathcal{Q}, \eta}_{\infty}$ is a cardinal in $\mathsf{CDM}^{+}(\mathcal{Q}, \eta)$.
If $\delta$ is regular in $\mathcal{V}$, then $\delta^{\mathcal{Q}, \eta}_{\infty}$ is a regular cardinal in $\mathsf{CDM}^{+}(\mathcal{Q}, \eta)$.
\end{cor}

While we know that it is possible that $\delta^{\mathcal{Q}, \eta}_{\infty}=\Theta^+$ in $\mathsf{CDM}^+(\mathcal{Q}, \eta)$, we still do not have an answer to the following question.

\begin{ques}\label{ques_on_delta_infty}
Is it consistent that $\delta^{\mathcal{Q}, \eta}_{\infty}>\Theta^+$ in $\mathsf{CDM}^+(\mathcal{Q}, \eta)$?
\end{ques}

We conjecture that some large cardinal assumption on $\delta$ in $\mathcal{V}$ gives an affirmative answer to \cref{ques_on_delta_infty}.
See \cref{conj_on_value_of_delta_infty}.

\end{document}